\documentclass[10pt]{article} 
\usepackage{changes}
\usepackage[greek,english]{babel}
\usepackage[utf8]{inputenc} 
\usepackage{upgreek}  
\usepackage{fancyhdr}
\usepackage{amsthm, amsmath, amssymb, mathtools} 
\usepackage[subnum]{cases}
\usepackage{bm}
\newcommand{\ara}{  \mathrm{a} }
\newcommand{\brb}{  \mathrm{b} }
\newcommand{\crc}{  \mathrm{c} }
\newtheorem{theorem}{ Theorem}
\newtheorem{lemma}{ Lemma}
\begin{document}
\allowdisplaybreaks
\title{Series expansion of weighted Finsler-Kato-Hardy inequalities}
\author{Konstantinos Tzirakis\\
\footnotesize Center for Mathematcal Modeling and Data Science, Osaka University, Japan\\
\footnotesize E-mail address: kostas.tzirakis@gmail.com} 
\date{}
\maketitle
\begin{abstract} In this work we consider weighted anisotropic Hardy inequalities and trace Hardy inequalities involving a general Finsler metric. We follow a unifying approach, by establishing first a sharp interpolation between them, extending the corresponding nonweighted version,  being established recently by a different approach. Then, passing to bounded domains, we obtain successive sharp improvements by adding remainder terms involving sharp weights and optimal constants, resulting in an infinite series-type improvement. The results extend, into the Finsler context, the earlier known ones within the Euclidean setting. The generalization of our results to cones is also discussed.
\end{abstract} {\bf Keywords:} Finsler-Laplacian, anisotropic Hardy-trace Hardy inequalities, sharp constants, infinite improvement
\section{Introduction} In this work, we consider anisotropic Hardy and trace Hardy inequalities, with the anisotropy being represented by a generic Finsler metric, which are associated to the so-called Finsler or anisotropic Laplacian; one of the most natural and foremost operators in anisotropic theory and Finsler geometry. Anisotropic variational problems arose in crystallograhy,  as minimization of anisotropic surface tensions for determining of equilibrium shapes of crystals; their study, was initiated in the historical work \cite{Wulff} using a geometrical construction. The relative mathematical theory were later developed, including analytic (see e.g. \cite{BianchiniCiraoloSalani}, \cite{CozziFarinaValdinoci_1}, \cite{CozziFarinaValdinoci_2}-\cite{DellaPietraGavitone_3}, \cite{FarinaValdinoci}, \cite{FeroneKawohlFinslerLaplacian}, \cite{Schaftingen}, \cite{WangXia_1}, \cite{WangXia_2}  and the references therein) and geometric aspects (see e.g. \cite{EspositoFuscoTrombetti}, \cite{FigalliMaggiPratelli}, \cite{HeLiMaGe} and the references therein), and attracts increasingly many attentions, due to its still further applications to other branches of physics, in biology and other fields.

Let  $n+\alpha>1, $ with $ \alpha\in (-1, 1),  $  and consider the upper half space  

$$ \mathbb{R}^{n+1}_+=\{z=(x,\,y)\in\mathbb{R}^{n+1}: x\in \mathbb{R}^n,\, y>0\}. $$
Starting with the Euclidean setting, the following weighted trace Hardy inequality,  \begin{equation} \label{ineq1}
H(n,\,\alpha)\int_{\mathbb R^n}\frac{u^2(x,0)}{|x|^{1-\alpha}}    \;dx
\leq \int_{\mathbb R^{n+1}_+} y^\alpha \, |\nabla u|^2\,dz, \quad \forall u\in C_0^\infty(\mathbb R^{n+1}), 
\end{equation} is well-known (see \cite{Herbst}) to hold with the best possible constant 

\begin{equation}\label{eq2}
H(n,\, \alpha)\,=\,(1-\alpha)\,\frac{ \Gamma^2(\frac{n+1-\alpha}{4}) \, \Gamma(\frac{\alpha + 1}{2})}
{\Gamma(\frac{3-\alpha}{2})\, \Gamma^2(\frac{n + \alpha -1}{4})},
\end{equation}
where $ \Gamma $ stands for the usual Gamma function defined as $ \Gamma(s)=\int_0^\infty t^{s-1}e^{-t}dt. $ It is also well known that the constant $ H(n,\, \alpha) $ is not attained in the homogeneous Sobolev space $ D^{1,2}({\mathbb R}^{n+1}_+, y^\alpha dz), $  defined as the completion of $  C_c^\infty(\overline{{\mathbb R}^{{n+1}}_+ }) $ with respect to the norm $ |\!| u |\!|_{D^{1,2}( {\mathbb R}^{n+1}_+, y^\alpha dz)}= ( \int_{{\mathbb R}^{n+1}_+} { y^\alpha}|\nabla u|^2\,dz )^{1/2}. $ We point out that inequality \eqref{ineq1} fails for $ |\alpha|\geq1. $ The power-type weight has a special significance, as \eqref{ineq1} can be translated into the sharp fractional Hardy inequality (cf. \cite{Tzirakis.FractionalHS}), via the characterization of the fractional Laplacian as a so called Dirichlet to Neumann map  \cite{Caffarelli.an.extensio.problem}.

The same situation holds for the following weighted Hardy inequality in the upper half space

\begin{eqnarray} \label{ineq3}
\frac{(n+ \alpha - 1)^2}{4}\int_{\mathbb R_+^{n+1} }\frac{y^\alpha u^2(z)}{|z|^{2 }}\,dz \leq
\int_{ \mathbb R_+^{n+1}  } y^\alpha |\nabla u|^2\,dz
,\;\;\;
 \forall u\in C_0^\infty({\mathbb R}^{n+1}).
\end{eqnarray}The constant $( n  + \alpha - 1)^2/4\, $ is the best possible, but it is not attained 
in $ D^{1,2}({\mathbb R}^{n+1}_+, y^\alpha dz). $

In the non weighted case, $ \alpha=0, $ inequality \eqref{ineq1} reduces to Kato's inequality and  \eqref{ineq3} to the classical Hardy inequality, which are ones of the well known mathematical formulations of the uncertainty principle in Quantum Mechanics, in the relativistic and non relativistic case respectively.  They are of fundamental importance in many branches of mathematical analysis, geometry and mathematical physics,  and they have been extensively studied, including several extensions and improvements. Relative to our interest in this work, we refer to \cite{Alvinoetal}, \cite{Tzirakis.Improving} for the sharp interpolation between \eqref{ineq1}, \eqref{ineq3}, in the non-weighted and weighted version respectively. 

Some of the Hardy inequalities have been also extended into the Finsler-metric context; see \cite{BrascoFranzina}, \cite{DellaPietraBlasioGavitone}, \cite{Schaftingen} and the references therein. To motivate our discussion below, we start with the inequality which interpolates between a Finsler-Kato and a Finsler-Hardy inequality (see the recent work \cite{AlvinoTakahashi}),  and asserts that,  for $ 2 \leq b < n+1, $

\begin{eqnarray} \label{ineq4}
C(n,b)\int_{  \mathbb R^{n} }\frac{u^2(x,0)}{H_0(x)}\,dx
+\frac{(b-2)^2}{4}\int_{   \mathbb R_+^{n+1} } \frac{u^2(z)}{\Phi_0^2(z)}\,dz
\leq
\int_{   \mathbb R_+^{n+1} }\Phi^2(\nabla u)\,dz,\qquad\forall u\in C_0^\infty(\mathbb R^{n+1}),
\end{eqnarray} where $ \Phi $ is a Finsler norm in $ \mathbb R^{n+1},  H_0(x)=\Phi_0(x, 0), $ 
 and $ \Phi_0(\cdot) $ is the anisotropic distance to the origin with respect to its dual norm (see \S\ref{subsection.2.1} for the precise assumptions and definitions). The optimal constant appearing in \eqref{ineq4} is \begin{flalign} \label{eq5}
C(n,b)=2 \frac{\Gamma(\frac{n-b+3}{4}) \, \Gamma(\frac{n+b-1}{4})}
{\Gamma(\frac{ n+1-b}{4}) \,\Gamma(\frac{n+b-3}{4}) }.  \end{flalign} For $ b=2, $   inequality \eqref{ineq4} reduces to the  Finsler-Kato inequality

\begin{equation}\label{ineq6} 2\frac{\Gamma^2(\frac{n+1}{4})}{\Gamma^2(\frac{n-1}{4})} \int_{  \mathbb R^{n} }\frac{u^2(x,0)}{H_0(x)}\,dx \leq \int_{   \mathbb R_+^{n+1} }\Phi^2(\nabla u)\,dz, \end{equation} with the constant in the left hand side being the best possible, while as $ b \rightarrow n+1, $ inequality \eqref{ineq4} reduces to the sharp Finsler-Hardy inequality (\cite{Schaftingen})

\begin{equation}\label{ineq7} \frac{(n-1)^2}{4}\int_{   \mathbb R_+^{n+1} } \frac{u^2(z)}{\Phi_0^2(z)}\,dz \leq \int_{   \mathbb R_+^{n+1} }\Phi^2(\nabla u)\,dz.
\end{equation}  Our first partial result is the anisotropic counterpart of \eqref{ineq1}:

\begin{equation} \label{ineq8}
H(n,\alpha)\int_{\mathbb R^n}\frac{u^2(x,0)}{ H_0^{1-\alpha}(x)}\,dx \leq \int_{\mathbb R^{n+1}_+} y^\alpha \,  \Phi^2(\nabla u)\,dz,\qquad \forall u\in C_0^\infty({\mathbb R}^{n+1}). 
\end{equation} The optimal constant $ H(n,\,\alpha) $ is given in \eqref{eq2}.   For $ \alpha = 0 $ inequality \eqref{ineq8} 
reduces to \eqref{ineq6}, and the inequality fails if $|\alpha|\geq 1. $ 

Actually, we will establish a sharp interpolation between the weighted Finsler-Kato  inequality \eqref{ineq8} and  the following weighted version of Finsler-Hardy inequality \eqref{ineq7}, \begin{eqnarray} \label{ineq9}
\frac{(n + \alpha -1)^2}{4}\int_{ \mathbb R_+^{n+1} }\frac{y^\alpha u^2(z)}{\Phi_0^2(z)}\,dz \leq
\int_{ \mathbb R_+^{{n+1}} } y^\alpha \Phi^2(\nabla u)\,dz
,\;\;\;  \forall u\in C^\infty_0(\mathbb R^{{n+1}}). \end{eqnarray} The result is stated in the following theorem. \begin{theorem}[Sharp interpolation of weighted Finsler-Kato-Hardy inequalities]
\label{Theorem1}
Let $ \alpha\in(-1,1),\, 2-\alpha \leq b<n+1. $ 
Then for all $u\in C_0^\infty(\mathbb R^{n+1}), $ the following inequality holds 

\begin{eqnarray} \label{ineq10}
K(n,\alpha, b)\int_{\mathbb R^n} \frac{u^2(x,0)}{H_0^{1-\alpha}(x)}\,dx +  \frac{(\alpha + b -2)^2}{4} \int_{ \mathbb R^{{n+1}}_+ } \frac{ y^\alpha u^2(z)}{\Phi_0^2(z)}\,dz
\leq \int_{ \mathbb R^{{n+1}}_+ } y^\alpha \Phi^2(\nabla u)\,dz,
\end{eqnarray} where 

\begin{equation} \label{eq11}
K(n, \alpha, b) \, = \,(1-\alpha)\,
\frac{ \Gamma(\frac{n -2\alpha - b+3 }{4}) \, \Gamma(\frac{n+b-1}{4}) \, \Gamma(\frac{ \alpha+1}{2})}
{\Gamma(\frac{3-\alpha}{2}) \, \Gamma(\frac{ n+1-b}{4}) \, \Gamma(\frac{n + 2\alpha +b - 3}{4})}.
\end{equation} The constant $ K(n, \alpha, b) $ is optimal. \end{theorem} Let us point out explicitly that  $ K(n, \alpha, b) = H(n, \alpha), $ when $ b = 2-\alpha$ and 
$ K(n, \alpha, b)\rightarrow 0,\,  $ as $ b\rightarrow n+1. $ This means that if $ b = 2 - \alpha, $ then inequality (\ref{ineq10}) reduces to \eqref{ineq8}, while as $ b \rightarrow n+1 $ then inequality  (\ref{ineq10}) reduces to \eqref{ineq9}. Note also that one can deduce \eqref{ineq10}, simply considering a convex combination of  \eqref{ineq8} and \eqref{ineq9}, however the constants obtained by this argument are not in general optimal. The optimal constant is not attained in the space of functions for which the right hand side of \eqref{eq11} is finite, yet no $L^p$-improvement of \eqref{eq11} is possible; see Section \ref{section.3} for a precise statement and proof of this fact. The non-weighted version of \eqref{ineq10}, i.e. for $ \alpha=0, $ has been proved in the recent work \cite{AlvinoTakahashi} by adopting in the Finslerian context a classical method in the Calculus of Variations introduced by Weierstrass. For the proof of the weighted version \eqref{ineq10} we will follow a different argumentation, which further leads to a refined version with sharp remainder terms, when we restrict the attention to functions supported on a bounded domain.  

More precisely, let us consider a bounded domain $ U $ containing the origin. Passing from the whole of the halfspace into the bounded domain $ U^+:= \mathbb R^{n+1}_+ \cap U, $ the Finsler-Hardy inequality \eqref{ineq9} reads 
\begin{eqnarray} \label{ineq12} 
\frac{(n + \alpha -1)^2}{4}\int_{U^+ }\frac{y^\alpha u^2 (z)}{ \Phi_0^2(z)}\,dz \leq
\int_{U^+} y^\alpha  \Phi^2(\nabla u)\,dz,\;\;\; \forall u\in C^\infty_0(U),
\end{eqnarray} with the same optimal constant $ ({n +\alpha-1})^2/4\, $ as in \eqref{ineq9}, due to the scaling invariance of \eqref{ineq7}. By approximation it follows that inequality \eqref{ineq12}  still holds with the same optimal constant for all functions $ u $ in the space $ D^{1,2}(U^+, y^\alpha dz), $  defined as the completion of $  C_{0}^\infty(\overline{\mathbb R^{n+1}_+}\cap U) $ with respect to the norm 
$ |\!| u |\!|_{D^{1,2}(U^+, y^\alpha dz)}= ( \int_{U^+} y^\alpha |\nabla u|^2\,dz)^{1/2}. $  It is clear that the constant $ (n+\alpha-1)^2/4 $ is not attained in $ D^{1,2}(U^+, y^\alpha dz), $  but contrary to (\ref{ineq9}),{ it is possible to improve \eqref{ineq12} , by adding $L^p$-norms in the left hand side. Indeed,  we will give below a sharp estimate of the correction term in \eqref{ineq12} . 
Similarly, if restrict the attention to the test functions supported on $U, $ then the trace Hardy inequality \eqref{ineq8} reads

\begin{eqnarray} \label{ineq13}
H(n, \alpha)\int_{U_0}\frac{u^2(x,0)}{H_0^{1-\alpha}(x)}\,dx\leq
\int_{U^+}{y^\alpha}{\Phi^2(\nabla u)}\,dz
,\quad \forall u\in C_0^\infty(U),
\end{eqnarray}
where, for notational convenience,  we abbreviate the boundary's portion $ \{x:(x,0)\in U\}$ to $ U_0. $ The  optimal constant $ H(n, \alpha) $ is the same as in \eqref{ineq8},  which clearly is not achieved in $D^{1,2}(U^+,y^\alpha dz).$ Again, driven by the fact that the best constant is not attained, we will show that we can improve \eqref{ineq13}, by adding sharp $L^p$-remainder terms in the left hand side. Actually, both \eqref{ineq12}  and \eqref{ineq13} will turn out to share the same sharp improvements, with the same best constants and optimal weights of the same singularity.

In our study of \eqref{ineq12} and \eqref{ineq13} we will follow a uniform approach by considering the sharp interpolation between them (cf. \eqref{ineq10}),

 \begin{eqnarray} \label{ineq14}
K(n,\alpha, b)\int_{U_0}
\frac{u^2(x,0)}{H_0^{1-\alpha}(x)}\,dx
 +  \frac{(\alpha + b -2)^2}{4} \int_{U^+ } \frac{ y^\alpha u^2(z)}{\Phi_0^2(z)}\,dz
\leq
 \int_{ U^+ } y^\alpha \Phi^2(\nabla u)\,dz,\quad \forall u\in C_0^\infty(U).
\end{eqnarray}

Then we will show that we can successively improve \eqref{ineq14} by adding to the left hand side lower order terms with
optimal weights and best constants. To properly state the result, for $ 0<\rho <1, $ we define recursively the functions 

\begin{equation*}
X_1(\rho)=\frac{1}{1-\ln \rho}, \;\;\; X_k(\rho) = X_1(X_{k-1}(\rho)), \;\;\; k=2,3,\ldots,\;\;\;\mbox{and we abbreviate}\;\; P_k(\rho) = X_1 X_2 \cdots X_k(\rho). 
\end{equation*}
Our result is summarized in the following theorem. \begin{theorem}\label{Theorem2} 
Let $ \,2 - \alpha \leq b < n+1, $  with $ \alpha \in (-1, 1). $
Then the following inequality is valid for all $ u\in  C_0^\infty(U), $  
\begin{eqnarray} \label{ineq15}
K(n,\alpha,b) \int\limits_{U_0}\frac{u^2(x, 0)}{H_0^{1-\alpha}(x)}\,dx +  \frac{(\alpha+b-2)^2}{4} \int\limits_{U^+} \frac{ y^\alpha \, u^2(z)}{\Phi_0^2(z)}\; dz + 
\frac{1}{4}\sum_{i=1}^\infty 
\int\limits_{U^+} \frac{ y^\alpha \,P^2_i}{\Phi_0^2(z)} \, u^2(z) \;dz \leq
\int\limits_{U^+} y^\alpha \, \Phi^2(\nabla u)\;dz,
\end{eqnarray} where the constant $ K(n,\alpha,b) $ is given in (\ref{eq11}) and 
$ P_i=P_i(\Phi_0(z)/D), $ with $ D:=\sup\limits_{z\in  U^+} \Phi_0(z). $ For any $ k=1,2,3,\cdots $ and fixed $ b, $ the constant $ \frac{1}{4} $ of the $k-$th remainder  term is optimal, i.e. $$ \frac{1}{4} = \inf\limits_{u\in  C_0^\infty(U)}\frac{\int\limits_{ U^+} y^\alpha  \Phi^2(\nabla u) \, dz\alpha, b)\int\limits_{U_0}
\frac{u^2(x,0)}{H_0^{1-\alpha}(x)}\,dx - \frac{(\alpha+b-2)^2}{4}\int\limits_{ U^+} \frac{  y^\alpha \,u^2(z)}{\Phi_0^2(z)} dz
-\frac{1}{4}\sum\limits_{i=1}^{k-1} \int\limits_{U^+} \frac{ y^\alpha\, P^2_i}{ \Phi_0^2(z)}\, u^2(z)\,dz}{\int\limits_{ U^+} \frac{ y^\alpha P^2_k}{\Phi^2(z)}\, u^2\,dz}. $$
Moreover, the power of the logarithmic weights in the remainder terms is optimal, in the sense that \eqref{ineq15} fails for more singular weights.
\end{theorem}

It is worth remarking about estimate \eqref{ineq15} for the non-weighted limiting instances $ b \to n+1 $ and $ b=2. $ For the special non-weighted Hardy case,  where $ \alpha = 0, $ $ b \to n+1, $ estimate \eqref{ineq15} reads 
\begin{equation}\label{ineq16}
\frac{(n-1)^2}{4} \int_{U^+} \frac{u^2(z)}{\Phi_0^2(z)}\,dz+\frac{1}{4}\sum_{i=1}^\infty \int_{ U^+} \frac{P^2_i}{\Phi_0^2(z)}\, u^2(z)\,dz \leq \int_{ U^+} \Phi^2(\nabla u)\,dz, \quad \forall u\in C_0^\infty(U),
\end{equation}
which extends, to the Finsler context, the earlier result \cite{Tertikasetal.Optimizing} where $ \Phi $ is the standard Euclidean metric. Let us also refer, at this point, to a sharp infinite series improvement obtained recently in \cite{DellaPietraBlasioGavitone}, regarding Hardy inequalities involving a general Finsler-distance to the boundary, and its earlier result \cite{BarbatisSeriesExpansion} in the Euclidean $L^p-$setting. The problem of series-type improving Hardy inequalities dates back, in the standard Euclidean context, by the question raised in the influential work \cite{Brezis.blowup}.

On the other hand, the non-weighted Kato case, i.e. for $ \alpha = 0 $ and $ b=2, $ estimate \eqref{ineq15} yields the following improvement of the Finsler-Kato inequality,

\begin{eqnarray}\label{ineq17}
2\frac{ \Gamma^2(\frac{n+1}{4})}
{\Gamma^2(\frac{n -1}{4})
 } \int_{U_0} \frac{u^2(x,0)}{H_0(x)}\,dx
+ \frac{1}{4} \sum_{i=1}^\infty 
\int_{U^+} \frac{P^2_i}{\Phi_0^2(z)}\, u^2(z) \,dz \leq  \int_{U^+}  \Phi^2(\nabla u) \,dz, \quad \forall u \in C_0^\infty(U).
\end{eqnarray}
In view of \eqref{ineq16}-\eqref{ineq17}, it turns out that both the Finsler-Hardy inequality and the Finsler-Kato inequality admit the same sharp series-type improvement.
Notice that the Finsler-Hardy inequalities follow from the classical ones, via the equivalence of norms,  however not with the best constants. To prove Theorems \ref{Theorem1}, \ref{Theorem2} we will adjust the argumentation \cite{Tzirakis.Improving} (see also \cite{Tertikasetal.Optimizing} for the case $ b \to n+1, \alpha=0$)  to the Finsler-metrics' setting. The approach is mainly based on an application of a Picone-type identity for the solutions of the Euler-Lagrange equations associated to the best constant for the interpolation inequality \eqref{ineq10} and its improved version \eqref{ineq15}, respectively. These solutions have not the right summability, thus they have no sense as minimizers, however by employing suitable perturbations we prove the optimality of the constants and the weights. Let us also note that we take advantage of the special structure of the variational problem, yielding $H_0$-symmetric (formal) minimizers, and thus leading to a dimension reduction of the problem, together with the fact that the $ H$-Finsler-Laplacian acts as a linear operator on $H_0$-radially symmetric smooth functions.

Let us finally note that an extension of the non-weighted interpolation inequality \eqref{ineq4} have been also derived recently in \cite{AlvinoTakahashi}, to the general case of a $\Phi$-cone with its vertex at the origin (see \S\ref{section.4} for the precise definition). A straightforward generalization of our argumentation may be applied to get also a sharp infinite series-type improvement on finite cones; see Theorem \ref{Theorem3} in Section \ref{section.4}.

The rest of this paper is organized as follows. In the preliminary Section \ref{section.2}, after giving a short reminder on Finsler norms, we introduce and survey the formal optimizer of \eqref{ineq10}, which will play a leading role to prove our results in the subsequent sections: In Section \ref{section.3}, we give the proof of Theorem \ref{Theorem1}, and the proof of Theorem \ref{Theorem2}  is given in Section \ref{section.4}, where we also discuss the extension of the results to the cones' case.
\section{Preliminaries}\label{section.2}
In this section we fix our basic notation, and collect the essential preliminaries that we will extensively use for the proof of Theorems \ref{Theorem1},   \ref{Theorem2}. More precisely, in \S \ref{subsection.2.1} we give a summary of what is needed about Finsler norms, and introduce the associated anisotropic Laplacian. Then, in \S \ref{subsection.2.2} we consider the Euler-Lagrange equations associated to the interpolation inequality  \eqref{ineq10}, and establish the key properties of the formal optimizer  (see \eqref{eq37} below) of the sharp constant $ K(n, \alpha, b). $
\subsection{Finsler norms}\label{subsection.2.1} To simplify our notation, in general we will use subscripts, though the customary usage of superscripts, especially in Finsler geometry.

Let the function $ H:\mathbb R^n \mapsto [0, \infty) $ be a norm of class $ C^2(\mathbb R^n\setminus\{0\}), $ with $ H^2 $ being strictly convex. In particular, $ H $ is positively homogeneous of degree 1, 

\begin{equation}\label{eq18}
H(\lambda \xi) = |\lambda| H(\xi),\qquad \forall \xi \in \mathbb R^n,\;\lambda \in \mathbb R.
\end{equation}
Differentiating \eqref{eq18} with respect to $ \xi $ yields that

\begin{equation}\label{eq19}
(\nabla H)(\lambda \xi) = \operatorname{sgn} \lambda\, \nabla H(\xi),\qquad \forall \xi \in \mathbb R^n\setminus\{0\},\;\lambda \in \mathbb R\setminus\{0\},
\end{equation}
while, differentiating \eqref{eq18} with respect to  
$ \lambda,  $ and then setting $ \lambda=1, $ we  get the radial directional derivative 

\begin{equation}\label{eq20}
\langle \xi, \nabla H(\xi)\rangle = H(\xi), \qquad \forall \xi \in \mathbb R^n\setminus\{0\},
\end{equation}
where $ \langle\cdot,\cdot\rangle $ denotes the standard Euclidean inner product. The equivalence with the standard Euclidean norm $ |\cdot|, $ reads, for some $ \gamma_1, \gamma_2 \in(0, \infty), $

\begin{equation}\label{ineq21}
\gamma_1 |\xi| \leq H(\xi)\leq \gamma_2 |\xi|,\qquad \forall \xi \in \mathbb R^n,
\end{equation}
where $ |\xi| = (\sum_{i=1}^n\xi_i^2)^{1/2}, \,\xi=(\xi_1,\ldots,\xi_n). $ 
The dual 
 norm of $ H $ is the function $ H_0:\mathbb R^n\mapsto [0,\infty) $ defined as

\begin{equation}\label{eq22}
H_0(x) := \sup\limits_{\xi\in\mathbb R^n \setminus \{0\} } \frac{\langle x, \xi \rangle}{H(\xi)},\qquad \forall x \in \mathbb R^n.
\end{equation}
Notice that $ H_0 $ is differentiable at all $ x \neq 0, $ positively homogeneous of degree 1, and the functions  $ H, H_0 $ are polar to each other in the sense that (cf. \cite[Corollary 7.4.2 and Theorem 15.1]{Rockafellar})  

\begin{equation}\label{eq23} 
H(\xi) = (H_0)_0(\xi)=\sup\limits_{x\in\mathbb R^n \setminus \{0\} } \frac{\langle x, \xi \rangle}{H_0(x)},\qquad \forall \xi \in \mathbb R^n. 
\end{equation}
Here as throughout the paper, we use the variable $ x $ to denote a point in
the ambient space $ \mathbb R^n $ equipped with the norm $ H_0, $ and the variable $ \xi $ for an element in its dual space, identified with $ \mathbb R^n $ endowed with $ H. $

It follows, directly from \eqref{ineq21}-\eqref{eq22}, that 

\begin{equation}\label{ineq24}
\frac{1}{\gamma_2} |x| \leq H_0(x)\leq \frac{1}{\gamma_1} |x|,\qquad \forall x \in \mathbb R^n,
\end{equation}
where $\gamma_1, \gamma_2 $ are the constants appearing in \eqref{ineq21}. 

The following connection between $ H $ and $ H_0 $ will be also useful (see e.g. \cite[Proposition 6.2]{Schaftingen}),

\begin{equation}\label{eq25}
H\bigl(\nabla H_0(x)\bigr) = 1,
\qquad 
 \forall x \in \mathbb R^n\setminus\{0\}.
\end{equation} An analogous property of \eqref{eq25}  holds true with the roles of $ H $ and $ H_0 $ being interchanged,  i.e.

\begin{equation}\label{eq26}
H_0\bigl(\nabla H(\xi)\bigr) = 1,\qquad \forall \xi \in \mathbb R^n\setminus\{0\},
\end{equation} as well as the following dual relation of \eqref{eq20}

\begin{equation}\label{eq27} 
\langle x, \nabla H_0(x)\rangle = H_0(x),\qquad \forall x \in \mathbb R^n\setminus\{0\}.
\end{equation} As a straightforward consequence of the definition \eqref{eq22}, we have the 
following Schwarz-type inequality

\begin{equation}\label{ineq28}
\langle x, \xi \rangle \leq H_0(x) H(\xi),\qquad \forall x, \xi \in \mathbb R^n,
\end{equation} with equality holding if $ x= \lambda H(\xi) \nabla H(\xi),  $ for some $ \lambda  \geq 0, $ as can be verified by \eqref{eq18}, \eqref{eq20}, \eqref{eq26}. 
From \eqref{eq20}, \eqref{eq23},  \eqref{eq25},   \eqref{eq26}, \eqref{eq27}, one finds that (cf. \cite[Lemma 2.2]{BellettiniPaolini})

\begin{equation}\label{eq29}
\nabla_{\!\xi} H\bigl(\nabla H_0(x)\bigr) = \frac{x}{H_0(x)},\qquad \forall x \in \mathbb R^{n}\setminus\{0\}.
\end{equation}

The values at any point of the gradient with respect to the $ x$-variables, of a function $ u,  $ denoted by $ \nabla_{\!x}u, $ 
 are considered as elements of $ \mathbb R^n $ endowed with $ H. $
The anisotropic Laplacian on $ \mathbb R^n,  $ associated  to  $ H, $ is  the operator defined by  (see e.g. \cite{FeroneKawohlFinslerLaplacian}) 

\begin{equation}\label{eq30}
\Delta_H u= \operatorname{div}_{\!x} \bigl(H(\xi)\nabla_{\!\xi} H(\xi)\bigr)\Big|_{\xi=\nabla_{\!x}  u}
\end{equation} where $ \operatorname{div}_{x} $ stands for the divergence operator with respect to the $x$-variables.  This operator extends the notion of the classical Laplacian to the anisotropic space $ \mathbb R^n $  equipped with a generic Finsler norm $ H. $ In the trivial case of $ H $ being the standard Euclidean norm $ H(\xi) = (\sum_{i=1}^n\xi_i^2)^{1/2},\, $ and more generally,  for a symmetric positive definite $ n\times n $ matrix $ A, $ 
 the norm $ H(\xi) = \sqrt{\langle A\xi, \xi\rangle},  $ the associated operator $ \Delta_H $ is linear. However in general,  the anisotropic Laplacian $ \Delta_H $ is a nonlinear operator, as in the typical example of $ H(\xi) = (\sum_{i=1}^n |\xi_i|^p)^{1/p},\; p>2. $ On account of \eqref{eq20}, \eqref{ineq21}, it is easy to deduce that $ \Delta_{ H}  $ satisfies the ellipticity condition

$$ \sum\limits_{i,j=1}^n \frac{\partial (H(\xi) H_{\xi_i}(\xi))}{\partial \xi_j} \xi_i\xi_j=H^2(\xi)\geq \gamma_1^2 |\xi|^2.$$ 

Occasionally in bibliography, a Finsler metric  is further required to be strongly convex, in the sense that the Hessian $ \nabla^2H^2(\xi) $ is positive definite at all $ \xi \neq 0. $ The weaker assumption of strict convexity is enough for our argumentation, without excluding from our results some interesting cases. %  
Let us also note that, an equivalent definition of $ H $ in geometric terms is also considered, as a non negative convex functional on a $ (n-1)$-dimensional sphere,  however being beyond the viewpoint of our analysis below.  

We will use the variable $ \zeta=(\xi,\mathrm{y}), $ with $ \mathrm{y} \in\mathbb R, $  to denote an element in the anisotropic space $ \mathbb R^{{n+1}} $ equipped with the norm

\begin{equation}\label{eq31}
\Phi(\zeta)=\sqrt{H^2(\xi)+\mathrm{y}^2}
\end{equation} and the variable $ z=(x,y), \, y \in\mathbb R,  $ for a point in the dual space $ \mathbb R^{{n+1}}, $ which turns out to be endowed with

\begin{equation*}
\Phi_0(z)=\sqrt{H_0^2(x)+y^2}.
\end{equation*}Let us point out that $ \Phi $ is a Finsler norm in $ \mathbb R^{{n+1}}, $ and in particular, it is straightforward to verify that \eqref{ineq28} jointly with the standard Schwarz inequality yield

\begin{equation}\label{ineq32}
\langle z, \zeta \rangle \leq \Phi_0(z) \Phi(\zeta)
,\qquad
 \forall z, \zeta \in \mathbb R^{n+1}.
\end{equation}  The energy functional 

$$ \int_{\mathbb R_+^{n+1}} y^\alpha  \Phi^2(\nabla u(z))\,dz $$ 
appearing in the right hand side of \eqref{ineq10}, is associated with the (weighted) $\Phi$-anisotropic Laplacian 

\begin{equation*} 
\Delta_{\Phi, \alpha} u(z) :=  \operatorname{div}\big(y^\alpha \,\Phi(\zeta) \,\nabla\Phi(\zeta)\bigr)\bigg|_{\zeta=\nabla u(z)} = y^\alpha  \Delta_H u(z)
+  y^{\alpha} u_{yy}(z) + \alpha y^{\alpha-1} u_{y}(z),
\end{equation*} where we denote by $  \operatorname{div}  $ the divergence operator with respect to the variables $ z=(x, y), $ that is 
$ \operatorname{div} \mathrm{F}:= \sum_{i=1}^n \frac{\partial F_i}{\partial x_i}+ \frac{\partial F_{n+1}}{\partial y}\, $ for a smooth vector field $  \mathrm{F}:z \mapsto (F_1(z),\ldots, F_{n+1}(z)). $ 
\\ \\
\subsection{Ground state}\label{subsection.2.2}
As already mentioned, a pivoting role in proving Theorems \ref{Theorem1},\ref{Theorem2}, will play the solution $ \psi $ (normalized up to a multiplicative constant) of the Euler-Lagrange 
equations associated to the interpolation inequality  \eqref{ineq10},

\begin{numcases}{\label{eq33}} 
\Delta_{\scriptscriptstyle \Phi,\alpha}\psi(z) + \frac{( \alpha + b-2)^2}{4}\frac{y^\alpha\psi(z)}{\Phi^2_0(z)} =0, 
\quad{z\in\mathbb R}_+^{n+1}, 
\label{eq33a}
\\
\lim\limits_{y\rightarrow 0^+} y^\alpha \frac{\partial \psi(x, y)}{\partial y}=-K(n, \alpha, b)\frac{\psi(x,0) }{H_0^{1-\alpha}(x)}
, \quad x \in \mathbb R^n\setminus\{0\}.
\label{eq33b}
\end{numcases} 
To simplify the notation, we will hereafter use the abbreviations

$$ \beta:=\frac{b - n - 1}{2}<0, \qquad\gamma:=\frac{2-\alpha- b}{2}\leq 0. $$
Looking at the special structure of problem \eqref{eq33}, we introduce the variables

\begin{equation} \label{eq34}
\rho(x,y) := \Phi_0(x, y)=\sqrt{H_0^2(x)+ y^2},\qquad
t(x,y) := \frac{y}{H_0(x)},
\end{equation} 
and we look for the solution in the form

$$
\psi(z)=\Psi\bigl(\rho(z), t(z), H_0(x)\bigr),
\qquad
\Psi(\rho, t, H_0)=\rho^\gamma\, H_0^\beta \, B(t).
$$
We directly check that

\begin{eqnarray}
\psi_{y}(z)
& =& 
 \frac{y}{\rho} \Psi_{\rho}
+
 \frac{1}{\phi} \Psi_{t},
\qquad
  \psi_{yy}(z)
\,=\,
 \frac{y^2}{\rho^2} \Psi_{\rho\rho}
+
 \frac{1}{H_0^2} \Psi_{tt}
+
 2\frac{t}{\rho} \Psi_{\rho t}
+ \, \Psi_{\rho}\left(  \frac{1}{\rho} -  \frac{y^2}{\rho^3}\right)
\label{eq35}
\\
\nabla_{\!  x}\psi(z) &=& \left(\frac{H_0}{\rho} \Psi_{\rho}  -\frac{t}{H_0} \Psi_{t} + \Psi_{H_0}  \right) \nabla H_0(x), \qquad \label{eq36}
\end{eqnarray}
and by \eqref{eq18}, \eqref{eq19},  \eqref{eq25}, \eqref{eq29} and \eqref{eq36}, we obtain that

\begin{equation*} 
H\bigl( \nabla_{\! x}\psi\bigr) \nabla_{\!\xi} H\bigl( \nabla_{\! x}\psi\bigr)
=
\left( \frac{1}{\rho} \Psi_{\rho}  -\frac{t}{H_0^2} \Psi_{t} + \frac{1}{H_0}\Psi_{\scriptscriptstyle H_0} \right) x.
\end{equation*}
Therefore by \eqref{eq27}, \eqref{eq35},  
and the facts that  $\, H_0^2+y^2=\rho^2 \,$ and $ \Delta_{\scriptscriptstyle \Phi,a}=  \alpha y^{\alpha-1}  \frac{\partial}{\partial y}  +  y^{\alpha} \frac{\partial^2}{\partial y^2} + y^\alpha\Delta_H, \, $   \eqref{eq33a} leads to the equation in the new variables \eqref{eq34}, 

\begin{eqnarray*} \Psi_{\rho \rho} + \frac{\rho^2}{H_0^4} \Psi_{tt} + \Psi_{\scriptscriptstyle H_0H_0} +  \frac{2 H_0}{\rho} \Psi_{\rho \scriptscriptstyle H_0} - \frac{2 t}{H_0}\Psi_{t\scriptscriptstyle H_0} +  \frac{n+\alpha}{\rho} \Psi_\rho + \left(  \frac{(3 - n) t}{H_0^2} + \frac{\alpha}{y H_0} \right) \Psi_t + \frac{n-1}{H_0}  \Psi_{\scriptscriptstyle H_0}
+ \frac{\gamma^2}{\rho^2} \Psi =0, \end{eqnarray*}
whence upon substituting $ \Psi, $ making straightforward computations and a normalization, we conclude that problem \eqref{eq33} admits a  solution of the form \begin{eqnarray}\label{eq37}
 \psi(z)&=& \Phi_0^\gamma(z) H_0^\beta(x) B\bigl(\frac{y}{H_0(x)}\bigr),\quad  x \in {\mathbb R}^n,\quad y\geq 0,\quad  z =  (x, y)\neq (0,0),
\end{eqnarray}
where $ B:[0,\infty) \mapsto \mathbb R$ satisfies the following boundary conditions problem

\begin{numcases}{\label{eq38}}
(t+t^3)\,B''(t)+\left[(4-b)t^2+\alpha\right]\, B'(t) + 
\frac{\beta (n+b-5) }{2}\, t \, B(t)=0,\quad  t>0, \label{eq38a} 
\\
B(0)=1,
\label{eq38b}
\\
\lim_{t\rightarrow\infty}t^{-\beta}B(t) \in\mathbb R. \label{eq38c}
\end{numcases} 
We point out that the initial condition \eqref{eq38b} is a normalization, which plays no essential role in the subsequent analysis. On the other hand, condition \eqref{eq38c} leads to a solution of \eqref{eq38} with the less possible singularity; see Appendix.  Notice also that $ \psi(x, y) $  is well defined in the set $\{(x,y): x=0,\, y> 0\}, $  due to the condition at the infinity. As for the notation, let us emphasize the dependence of $ B $ on $ n, \alpha, b, $ as we will see below, yet suppressing this dependence, for the sake of simplicity. 

In the specific case $ n=3 $ with $ \alpha = 0 $ and $  b=2, $  problem  
(\ref{eq38}) is solved explicitly and we have 
$ B(t) = 1-\frac{2}{\pi} \arctan(t). $ For the general case, we use the  transformation  $ \mathrm{z}=-t^2,\, $ which maps the regular singular points $\pm i,\, 0,\, \infty $ to $ \, 1,\, 0,\, \infty, $ respectively. Then, setting $\omega( \mathrm{z}) = B(t),\, $ 
problem (\ref{eq38}) is transformed to the following boundary conditions problem for the hypergeometric equation

\begin{numcases}{\label{eq39}}
\mathrm{z}\, (1- \mathrm{z}) \,\frac{d^2 \omega}{d  \mathrm{z}^2} \, +\, \left[ \frac{\alpha+1}{2} - \frac{5 - b}{2} \,  \mathrm{z} \right] \, \frac{d \omega}{d  \mathrm{z}} +  \frac{\beta (5-n-b)}{8} \,\omega( \mathrm{z})=0, \; \;
 -\infty <  \mathrm{z} < 0, \label{eq39a}
\\
\omega(0) = 1, \label{eq39b} 
\\
\lim_{ \mathrm{z} \rightarrow -\infty}(- \mathrm{z})^{-\beta/2}\omega( \mathrm{z}) \in\mathbb R.  \label{eq39c}
\end{numcases} 
The study of \eqref{eq39} for deriving the required properties of $ B,  $ is rather technical, due 
to the dependence on several relations of the parameters $\alpha, b, n. $ For readability's sake, we state here those properties of $ B $ that we need for the proof of Theorems \ref{Theorem1}-\ref{Theorem3}, deferring their complete proof until the Appendix; also, cf. \cite[Lemma 1]{Tzirakis.Improving}.

For notational convenience of the statements, we write in the sequel, $ g \sim h $ for two real functions $ g, h $ of the variable $ t $ (resp. $\, z$)  to mean that $ c_1 \,g  \leq h  \leq  c_2\, g, \,  $ for some positive constants $ c_1, c_2 $ independent of $ t$ (resp. $\,z$).  

It turns out (see Appendix) that problem \eqref{eq38} has a positive decreasing solution $ B, $ such that

\begin{equation}\label{eq40}
B  \sim   (1+t^2)^{\beta/2}\qquad\mbox{and}
\qquad B'  \sim  -t^{-\alpha} (1+t^2)^{\frac{\alpha+\beta-1}{2}},\qquad \forall t>0, 
\end{equation} with

\begin{equation}\label{eq41} 
 tB'-\beta B(t)=O(t^{\beta-2}),\quad\mbox{as}\;\; t \rightarrow\infty. 
\end{equation} We also have

\begin{equation}\label{eq42}
 \lim\limits_{t\rightarrow 0^+} t^\alpha B'(t) = - K(n,\alpha,b),
\end{equation} where  $ K(n, \alpha, b) $ is given in \eqref{eq11}. 

In view of \eqref{eq27}, and noting that \begin{equation}\label{eq43}
 \nabla\Phi_0(z)= \frac{1}{\Phi_0(z)}(H_0(x) \nabla H_0(x), \,y), 
\end{equation} and \begin{equation}\label{eq44}
 \nabla\psi=(M_1\nabla H_0,\, M_2), 
\end{equation} for the function $ \psi $ defined in \eqref{eq37},  
where 

$$
M_1:=\gamma  \Phi_0^{\gamma-2}  H_0^{\beta+1} B + \beta \Phi_0^\gamma H_0^{\beta - 1}  B  - \Phi_0^\gamma H_0^{\beta-2}  B' y,\;  
\qquad
 M_2:= 
 \gamma \Phi_0^{\gamma-2}  H_0^\beta   B y  
 + \Phi_0^\gamma H_0^{\beta-1} B', 
$$ it is straightforward to verify that 

\begin{equation}\label{eq45}
\langle \nabla \psi, z\rangle = \frac{1-\alpha-n}{2}\psi(z),\quad \forall z\in\mathbb R^{{n+1}}_+\setminus\{0\}.
\end{equation}As a consequence of   \eqref{eq40}-\eqref{eq42}, \eqref{eq45}, we get the following  uniform asymptotic behavior of the ground state $ \psi.$ 
\begin{lemma}\label{lemma1}
For the function $ \psi $ defined in \eqref{eq37}-\eqref{eq38}, there holds

\begin{equation}\label{rel46}
\psi  \sim  \Phi_0^{-\frac{n+\alpha-1}{2}},\;\,\mbox{in}\;\; \mathbb R_+^{n+1}.
\end{equation} Moreover, if $ \alpha \in (-1, 0], $ then $ \Phi (\nabla \psi ) \sim  \Phi_0^{-\frac{n+\alpha + 1}{2}}, $ and if $ \alpha \in (0, 1), $ then $ \Phi (\nabla \psi )   \sim \Phi_0^{-\frac{n + 1 - \alpha}{2}} y^{-\alpha},\,$ in $ \mathbb R_+^{n+1}. $
\end{lemma}\begin{proof}
The asymptotics \eqref{rel46} is a direct consequence of the definition \eqref{eq37} and the asymptotics of $ B $ in \eqref{eq40}. Then \eqref{eq45}  jointly with \eqref{ineq32} and \eqref{rel46} yields $ \Phi (\nabla \psi(z))   \geq c \,\Phi_0(z)^{-\frac{n+1 + \alpha}{2}}\, $ in $\mathbb R_+^{n+1}, $ for some constant $ c>0, $ independent of $ z. $ Moreover, in view of \eqref{eq25} and \eqref{eq44}, we have, for $  t=y / H_0(x), $ 

$$ \Phi^2(\nabla\psi(z)) = \gamma \left( \gamma+2\beta \right) \Phi_0^{2\gamma-2}(z) 
 H_0^{2\beta}(x)  B^2(t) + \Phi_0^{2\gamma}(z) H_0^{2\beta-2}(x) B'^2(t)  +  \Phi_0^{2\gamma}(z) H_0^{2\beta-2}(x) \left[t B'(t)  -\beta B(t) \right]^2 $$
whereafter  the asymptotics of $ \Phi(\nabla\psi) $ results upon an application of \eqref{eq40}-\eqref{eq42}.
\end{proof}
\section{Sharp interpolation}\label{section.3}
This Section is devoted to the proof of Theorem \ref{Theorem1}. We retain the notation introduced in the previous section, and let us introduce some further notation: We denote the unit ball, with respect to the $ H_0$-norm in $ \mathbb R^n, $ by $\,\mathcal{W}_1:=\{x\in \mathbb R^n : H_0(x)\leq 1\}, $\, the so-called Wulff shape (or equilibrium crystal shape) of $ H_0,  $ centered at the origin. 
 We will also follow the usual convention of denoting by $ c $ or $ C $  a general positive constant, possibly varying from line to line. Relevant dependencies on parameters will be emphasized by using parentheses or subscripts.  In particular,  we set $ \omega_{{\scriptscriptstyle H},n}:=P_{H_0}(\mathcal{W}_1;\mathbb R^n)=n|\mathcal{W}_1|,  $  the anisotropic $H$-perimeter of $ \mathcal{W}_1. $

A key tool for the proof of Theorem 
\ref{Theorem1} is the following generalization of Picone's identity in the Finsler setting,

\begin{eqnarray}\label{eq47}
y^\alpha f(u,\psi)=y^\alpha \Phi^2(\nabla u)+\frac{u^2}{\psi} \Delta_{\Phi, \alpha}\psi -\operatorname{div}(y^\alpha \frac{u^2}{\psi} \Phi(\nabla\psi)\nabla_{\!\zeta} \Phi(\nabla\psi))
\end{eqnarray} where

\begin{equation}\label{eq48}
f(u,\psi)  
:=
\Phi^2(\nabla u) 
+
\frac{u^2}{\psi^2} \Phi^2(\nabla\psi) 
-2\frac{u}{\psi}\bigl\langle \nabla u , \Phi(\nabla\psi)\nabla_{\! \zeta} \Phi(\nabla\psi)\bigr\rangle,
\end{equation} which can be easily checked by straightforward differentiation and taking into account \eqref{eq20} and the expression 

\begin{equation*}
\Phi(\nabla\psi) \nabla_{\!\zeta} \Phi(\nabla\psi)=(H(\nabla_{\! x}\psi)\nabla_{\!\xi} H(\nabla_{\! x}\psi),\, \psi_y).
\end{equation*}
Moreover, by \eqref{eq18}, \eqref{eq19},  \eqref{eq26}, \eqref{ineq32} is easily seen that 

\begin{equation}\label{eq49} 
f(u,\psi) 
\geq
\Phi^2(\nabla u) + \Phi^2\bigl(\frac{u}{\psi}\nabla\psi\bigr) - 2 \Phi(\nabla u) \Phi\bigl(\frac{u}{\psi}\nabla\psi\bigr)
= \big(\Phi(\nabla u) 
-
 \Phi\bigl(\frac{u}{\psi}\nabla\psi\bigr)\big)^2
\geq 0.
\end{equation}
 We are now ready to proceed with the
\begin{proof}
[\textbf{Proof of Theorem 
\ref{Theorem1}}]
We begin with the proof of \eqref{ineq10}, which by standard approximation with smooth cutoff functions, it suffices to be proved for $ u \in C_0^\infty(\mathbb R^{n+1} \setminus \{0\} ). $ 
Then, we integrate  \eqref{eq47}, 
apply the divergence theorem to the last term, and use equations \eqref{eq33}, to obtain
\begin{flalign} 
\int_{\mathbb R_+^{n+1}} y^\alpha f(u, \psi) \,dz
&=\int_{\mathbb R_+^{n+1}}  \,y^\alpha \,  \Phi^2(\nabla u)  \,dz +
\int_{\mathbb R_+^{n+1}}  \,\frac{u^2}{\psi}
\Delta_{\Phi, \alpha}\psi\, dz + \int_{ \mathbb R^n}  \, \lim\limits_{y\rightarrow 0^+} y^\alpha \,\frac{\partial \psi(x, y)}{\partial y} \frac{u^2}{\psi} \,dx \nonumber \\&=
\int_{\mathbb R_+^{n+1}}  \, y^\alpha \, \Phi^2(\nabla u)  \,dz-\frac{( \alpha + b - 2 )^2 }{4}
\int_{\mathbb R_+^{n+1}} \frac{y^\alpha\,u^2(z)}{ \Phi_0^2(z)} \,dz-K(n, \alpha, b)
\int_{\mathbb R^n}   \frac{u^2(x,0)}{H_0^{1-\alpha}(x)} \,dx. \,\;
\label{eq50}
\end{flalign}Let us here point out that on $ supp\, u, $ the function $ \psi $ is smooth and uniformly bounded by some positive constant. Hence the integrand $ u / \psi $ has not singularity at the origin; actually  $ u / \psi \in C^\infty_c(\overline{\mathbb R_+^{n+1}}\setminus\{0\}). $ 

Then, in view of \eqref{eq49}, we conclude that 

\begin{eqnarray*} 
K(n,\alpha,b)
\int_{\mathbb R^n } \frac{u^2(x, 0)}{ H_0^{1-\alpha}(x)} \,dx
+
\frac{ ( \alpha + b - 2 )^2 }{4}
\int_{ \mathbb R_+^{n+1} } 
\frac{y^\alpha\, u^2(z)}{\Phi_0^2(z)} \,dz
\leq \int_{ \mathbb R_+^{n+1} } y^\alpha \, \Phi^2(\nabla u)  \,dz,
\end{eqnarray*}
for all $  u\in C^\infty_0(\mathbb R^{n+1}\setminus \{0\}) $ and therefore by approximation, for all $u\in C^\infty_0(\mathbb R^{n+1}). $ 

In order to verify the optimality of the constant $K(n, \alpha, b), $ we define in $  D^{1,2}( {\mathbb R}^{n+1}_+,\, y^\alpha\,dz) $ the quotient

\begin{equation*}
Q[u]:=\frac{\int_{ \mathbb R_+^{{n+1}} } y^\alpha \, \Phi^2(\nabla u)\,dz 
-\frac{(\alpha + b -2)^2}{4}
\int_{ \mathbb R_+^{{n+1}} } \frac{y^\alpha u^2(z) }{ \Phi_0^2(z)}\,dz}{\int_{\mathbb R^n} \frac{u^2(x, 0) }{H_0^{1-\alpha}(x)}\,dx}=\frac{N[u]}{D[u]},
\end{equation*}
and we will show that there exist functions $u_\epsilon \in D^{1,2}(  {\mathbb R}^{n+1}_+,\, y^\alpha\,dz) $  
such that $ \lim\limits_{\epsilon\rightarrow 0^+}
Q[u_\epsilon]= K(n,\alpha,b). $

To this end, we consider the $\Phi_0$-cylinders $ {\mathcal C}_1:=\{(x, y)\in {\mathbb R}^{n+1}  : H_0(x) <  1,\, |y| < 1 \},\;  {\mathcal C}_2:=2\mathcal C_1, $ and a smooth cutoff function $\eta\in C^1_0({\mathcal C}_2), $ such that 
$\eta\equiv 1 $ in $ {\mathcal C}_1. $
We then take the following approximations to the ground state $\psi $ (see \eqref{eq37})
\begin{eqnarray*}
u_\epsilon(x, y)= \begin{cases} \eta(z)\,\psi(x, y), &  y\geq \epsilon,  \\ \eta(z)\,\psi(x,\epsilon) , & 0\leq y \leq  \epsilon. \end{cases} \end{eqnarray*}
Let us next determine the leading asymptotic behavior of  $ D[u_\epsilon] $ and $ N[u_\epsilon], $ as $ \epsilon\to 0. $  For both terms we will employ the ``$H_0$-polar coordinates'': $ x = r  \mathrm{w}, \, $ with  $ r = H_0(x), \, \mathrm{w}\in \partial\mathcal{W}_1. \, $ As regards  the denominator $ \, D[u_\epsilon], $ we split the integration ``near'' and ``away'' the singularity, and in view of \eqref{eq37} together with the asymptotics of $ B $ in \eqref{eq40}, we have

\begin{eqnarray} 
 D[u_\epsilon]
&=&
 \int_{\mathcal{W}_1}\frac{\psi^2(x,\epsilon) }{H_0^{1-\alpha}(x)}  \;dx 
+
\int_{\mathcal{W}_2\setminus\mathcal{W}_1}
\!\!\!
\frac{\eta^2(x, 0)\,\psi^2(x,\epsilon)}{H_0^{1-\alpha}(x)}  \;dx
  = 
\omega_{{\scriptscriptstyle H}, n}
\int_{0}^1 \left(1 + \frac{\epsilon^2}{r^2}\right)^\gamma B^2\left(\frac{\epsilon}{r}\right)\frac{1}{r}\,dr + O(1)
\nonumber\\
&=& 
\omega_{{\scriptscriptstyle H}, n}
 \int_\epsilon^{\infty} (1+s^2)^\gamma B^2(s)\frac{1}{s} \, ds + O(1), \qquad\mbox{as $\epsilon\rightarrow 0^+. $ } \label{eq51}
\end{eqnarray}We proceed now to estimate the numerator $ N[u_\epsilon]. $ First, notice that in view of \eqref{eq37} together with the asymptotics of Lemma \ref{lemma1}, it is straightforward to check that away from the singularity,

\begin{equation}\label{eq52} 
\int_{ {\mathcal C}_2\setminus \mathcal{C}_1   }
y^\alpha 
\bigg( 
\Phi^2(\nabla u_\epsilon)
- \gamma^2
 \frac{  u_\epsilon^2(z) }{\Phi_0^2(z)}
\bigg)\, dz
= 
O(1),\quad\mbox{as}\;\;\epsilon\to 0.
\end{equation}Moreover, using \eqref{eq25} and \eqref{eq40}, we get, for some constant $ C_{{\scriptscriptstyle H},n,a,b}>0, $ independent of  $ \epsilon, $

$$ \int_0^\epsilon y^\alpha 
\int_{ \mathcal{W}_1 } \Phi^2(\nabla u_\epsilon) \, dx\,dy \leq  C_{{\scriptscriptstyle H},n,a,b} I_1(\epsilon) \, + \, I_2(\epsilon), $$ 
where \begin{eqnarray*}
I_1(\epsilon)
 := \int_0^\epsilon y^\alpha 
\int_{ \mathcal{W}_1 } 
\Phi_0^{2\gamma+2\beta-2}(x,\epsilon)
 \, dx\,dy  = \frac{ \epsilon^{\alpha +1} }{  \alpha + 1 } 
\, \omega_{{\scriptscriptstyle H},n} 
\int_0^1 
\frac{r^{n-1}}{( r^2 + \epsilon^2 )^\frac{n+\alpha+1}{2}}
\, dr 
 = 
O(1), \;\;\;\mbox{as} \;\; \epsilon\rightarrow 0, \end{eqnarray*}
as well as, for some positive constants  $ C_1, C_2, $ independent of  $ \epsilon, $

\begin{eqnarray}\label{eq53}
I_2(\epsilon)& := &
\int_0^\epsilon y^\alpha 
\int_{ \mathcal{W}_1 } \Phi_0^{2\gamma}(x,\epsilon)
H_0^{2\beta-2}(x)
\Big[ \frac{ \epsilon }{H_0(x)}
B'\bigl(\frac{ \epsilon }{H_0(x)}\bigr)
-\beta B\bigl(\frac{ \epsilon }{H_0(x)}\bigr)
\Big]^2  \, dx\,dy
\nonumber \\
&=& \frac{\omega_{{\scriptscriptstyle H},n}}{\alpha+1} 
\int_\epsilon^\infty \big(1+s^2)^{\gamma}
s^\alpha  \Big[s B'(s) - \beta B(s) \Big]^2 \, ds \leq C_1  \int_\epsilon^1  s^\alpha\,ds + C_2 
\int_1^\infty s^ {\alpha+2\beta-4} \, ds 
=O(1), 
\end{eqnarray}
as $ \epsilon  \to 0,\, $ where, in addition, we used \eqref{eq41}. We may also exploit the asymptotics \eqref{eq40}-\eqref{eq41}, to show that

\begin{equation}\label{eq54}
\int_0^\epsilon\int_{ \mathcal{W}_1 }
y^\alpha 
 \frac{  u_\epsilon^2(z) }{\Phi_0^2(z)}\, dx\,dy
= 
O(1),\quad\mbox{as}\;\;\epsilon\to 0.
\end{equation}
From \eqref{eq52}-\eqref{eq54}, we conclude that

\begin{equation}\label{eq55}
N[u_\epsilon]
= \int_\epsilon^1 
\int_{\mathcal{W}_1}
  y^\alpha \bigg( 
 \, \Phi^2(\nabla  \psi)  - \gamma^2\frac{  \psi^2(x, y)}{\Phi_0^2(x, y)} \bigg) \,dx\, dy
\,+\,
 O(1),\quad\mbox{as}\;\;\epsilon\to 0.
\end{equation}
To estimate the integral in the right hand side of (\ref{eq55}), we use (\ref{eq25}) and (\ref{eq44}),  then we employ again the ``$H_0$-polar coordinates'': $ x = r  \mathrm{w}, \, $ with  $ r = H_0(x), \, \mathrm{w}\in \partial\mathcal{W}_1,\, $ and make the change of variable $ s=y/r, $ to deduce that

\begin{multline}\label{eq56}
\int_\epsilon^1 \int_{\mathcal{W}_1} y^\alpha \bigg(\Phi^2(\nabla  \psi)  - \gamma^2\frac{  \psi^2(x, y)}{\Phi_0^2(x, y)} \bigg)\,dx\, dy =\; \omega_{{\scriptscriptstyle H}, n} \int_\epsilon^1 \frac{1}{y} \int_y^\infty s^\alpha (1+s^2)^\gamma \Big[\beta^2 \,B^2(s) \, + \,2 \, \beta \,\gamma \, (1+s^2)^{-1} \,B^2(s) \\ +(1+s^2)\,B'^2(s) - 2\, \beta  \,s\, B(s)\, B'(s) \Big] \,ds\,dy. \end{multline}
We express the $ds$-integral in \eqref{eq56} as the limit,  for $ R\to\infty, $ of the sum

\begin{equation}\label{eq57}
\int_y^R
s^\alpha (1+s^2)^\gamma \big[
\beta^2 \,B^2(s) \,+\,2 \, \beta \,\gamma \, (1+s^2)^{-1} \,B^2(s)
\big]
\,ds\,   + \, I(y;R)
\end{equation} $$ I(y;R):=  \int_y^R
s^\alpha  (1+s^2)^{\gamma+1} \,B'^2(s) \,- \, 2 \, \beta s^{\alpha+1}\,\left(1+s^2\right)^{\gamma} \, B(s)B'(s)
\,ds, $$
and then we use integrations by parts to estimate the terms in $ I(y;R). $ More precisely, for the first term, we  multiply (\ref{eq38a}) by $ t^{\alpha-1}(1+t^2)^\gamma B, $ to get
\begin{equation}\label{eq58}
(s^\alpha (1+s^2)^{\gamma + 1}B'(s))' B(s) +  \frac{\beta (n+b-5)}{2}\, 
s^\alpha \, (1+s^2)^\gamma B^2(s)=0,\qquad s>0,
\end{equation} whereupon an integration by parts yields 

\begin{equation}\int_y^R s^\alpha  (1+s^2)^{\gamma+1} \,B'^2(s) \,ds=\frac{ \beta (n+b-5)}{2}
\, \int_y^R s^\alpha \, (1+s^2)^\gamma B^2(s) \,ds + \Big[s^\alpha\left(1+s^2\right)^{ \gamma+1}
B'(s)  B(s) \Big]_{s=y}^{s=R},\label{eq59}
\end{equation}
and for the second term, we have

\begin{flalign}
&\int_y^R
s^{\alpha+1}\,\left(1+s^2\right)^{\gamma} \, B(s) B'(s)
\,ds = \frac{1}{2}\int_y^R s^{\alpha+1}\,\left(1+s^2\right)^{\gamma}\left(B^{2}(s)\right)'
\,ds = \nonumber \\ &
-\frac{(\alpha+1)}{2}\int_y^R
s^\alpha \left(1+s^2\right)^{\gamma} B^2(s)
\,ds - \gamma \int_y^R s^{\alpha+2}  \left(1+s^2\right)^{\gamma-1} B^2(s)
\,ds \,+\, \frac{1}{2} \Big[s^{\alpha+1} \left(1+s^2\right)^{\gamma} B^2(s)\Big]^{s=R}_{s=y}.
 \label{eq60}
\end{flalign}
Substituting (\ref{eq59})-(\ref{eq60})  in (\ref{eq57}),  then letting $ R \to \infty, $ taking into account  \eqref{eq41}, and eventually substituting in \eqref{eq56}, it follows that

\begin{eqnarray}\label{eq61}
N[u_\epsilon]=
- \omega_{{\scriptscriptstyle H}, n}
\int_\epsilon^1 
y^{\alpha-1}\left(1+ y^2\right)^\gamma\,
B(y)\,
\Big[ 
B'(y) \,+\, y\, ( y\, B'(y)\,  - \beta \, B(y)) 
\Big]  \, dy 
+ O(1), \;\;\;\mbox{as}\; \epsilon\rightarrow 0.
\end{eqnarray}
In view of the asymptotics of $ B' $ in  (\ref{eq40}), the integrals in (\ref{eq51}) and (\ref{eq61}), are unbounded as $ \epsilon\to 0^+, $  then we apply L'H\^opital 's rule and use (\ref{eq38b}) and (\ref{eq42}), to conclude 

\begin{eqnarray*}
\lim\limits_{\epsilon\rightarrow 0^{\scriptscriptstyle+}}
Q[u_\epsilon]
&=&\lim\limits_{\epsilon\rightarrow 0^{\scriptscriptstyle+}}
\frac{
-\omega_{{\scriptscriptstyle H}, n}
\int_\epsilon^1 
 s^{\alpha - 1} \left(1+s^2\right)^\gamma\,
B(s)\,[ B'(s) + s ( s B'(s) - \beta  B(s))  ] \,ds + O(1)
}
{  {\omega_{{\scriptscriptstyle H}, n}}
\int_\epsilon^{\infty} (1+s^2)^\gamma B^2(s)\frac{1}{s}\,ds + O(1)}
\\
&=&
\lim\limits_{\epsilon\rightarrow 0^{\scriptscriptstyle+}}
\frac{ 
-\epsilon^{\alpha} \,(1+\epsilon^2)\,
B'(\epsilon)
}{B(\epsilon)} + \beta \epsilon^{ \alpha+1} =
K(n, \alpha, b).
\end{eqnarray*}
The proof of Theorem 
\ref{Theorem1} is now completed.
\end{proof} Notice that, in view of (\ref{eq48})-(\ref{eq50}), the best constant $ K(n, \alpha, \beta) $ in (\ref{ineq10}), is not attained in  $ D^{1,2}({\mathbb R}^{n+1}_+, y^\alpha dz). $ The fact that the best constant is not attained suggests that one might look for an improvement of (\ref{ineq10}), by adding a correction term in the least hand side. In the next section, we indeed establish $L^2$-improvements for the case of bounded domains. However, before proceed to this case, let us point out that there is no improvement of (\ref{ineq10}) in the usual $L^p$-sense; that is there are no constant $ C>0, $ nontrivial potential $ V \geq 0$ and exponent $ p > 0 $ such that

\begin{eqnarray*}
 C\left(\int_{ \mathbb R_+^{{n+1}} } V(z)|u(z)|^p\,dz\right)^\frac{2}{p}\leq
\int_{ \mathbb R_+^{{n+1}} } y^\alpha\,  \Phi^2(\nabla u)  \, dz
- K(n,\alpha, b)\int_{\mathbb R^{n}}
\frac{u^2(x,0)}{ H_0^{1-\alpha}(x)   }\,dx
 -  \frac{(\alpha + b -2)^2}{4}\int_{ \mathbb R_+^{{n+1}}} \frac{ y^\alpha\, u^2(z)}{  \Phi_0^2(z) }\,dz,
\end{eqnarray*} for all $ u  \in D^{1,2}({\mathbb R}^{n+1}_+,  y^\alpha dz). $ This fact can be demonstrated by taking the following perturbations of the ground state $ \psi $ defined in \eqref{eq37},
\begin{eqnarray*}
u_\varepsilon(z) = \begin{cases} 
\psi(z)  \Phi_0^{\varepsilon/2}(z), & \Phi_0(z) \leq 1,  
\\
\psi(z) \Phi_0^{-\varepsilon/2}(z), &   \Phi_0(z) \geq 1,
\end{cases} 
\end{eqnarray*} for $ \varepsilon > 0, $ and then showing that 

\begin{eqnarray*}
\frac{\int_{ \mathbb R_+^{{n+1}} } y^\alpha \Phi^2(\nabla u_\varepsilon) \, dz 
- K(n, \alpha, b)\int_{\mathbb R^{n}} \frac{u_\varepsilon^2(x,0)}{ H_0^{1-\alpha}(x)}\,dx
 -  \frac{(\alpha + b -2)^2}{4}\int_{ \mathbb R_+^{{n+1}} } \frac{  y^\alpha u_\varepsilon^2(z)}{\Phi_0^2(z) }\,dz  }
{\left(\int_{ \mathbb R_+^{{n+1}} } V(z)|u_\varepsilon|^p\,dz\right)^\frac{2}{p}}\xrightarrow{ \varepsilon \to 0} 0.
\end{eqnarray*} To prove this fact, we notice first that  

$$ \begin{cases} 
\Delta_{\scriptscriptstyle \Phi,\alpha}u_\varepsilon(z) + \big[
 \frac{( \alpha + b-2)^2}{4} - \frac{\varepsilon^2}{4}  
\big] \frac{y^\alpha u_\varepsilon(z)}{\Phi^2_0(z)} =0, 
&{z\in\mathbb R}_+^{n+1}, \; \Phi_0(z) \neq 1, \\
\lim\limits_{y\rightarrow 0^+} y^\alpha \frac{\partial u_\varepsilon(x,  y)}{\partial y}=-K(n, \alpha, b)\frac{u_\varepsilon(x, 0) }{H_0^{1-\alpha}(x)}, & x \in  \mathbb R^n\setminus\{0\}, \; H_0(x) \neq 1,
\end{cases} $$ as is seen by straightforward computations, as in Section \ref{subsection.2.2}. 
Moreover, from \eqref{eq18}, \eqref{eq19}, \eqref{eq25}, \eqref{eq27}, \eqref{eq29}, \eqref{eq45} we get that $$
\Phi(\nabla u_\varepsilon) 
\big\langle 
 \nabla_{\scriptscriptstyle\!\zeta} \Phi(\nabla u_\varepsilon),\,
 \nabla \Phi_0(z) 
\big\rangle  = \begin{cases} 
\frac{1-n-\alpha +\varepsilon}{2 \, \Phi_0(z)} \; u_\varepsilon, & \Phi_0(z) < 1,  \\ \frac{1- n -\alpha - \varepsilon}{2 \, \Phi_0(z)} \; u_\varepsilon, & \Phi_0(z) > 1, 
\end{cases}
$$
so taking into account the asymptotics  (\ref{rel46}), and noting  by \eqref{ineq21}, \eqref{eq25}, that $ 1=\Phi(\nabla\Phi_0)\sim |\nabla \Phi_0| $ we deduce that,  for some positive constant $ C_{{\scriptscriptstyle H},n,\alpha,b}, $  \begin{eqnarray*}
 \bigg|
\int_{(\partial W_R)^+ } 
y^\alpha \,u_\varepsilon \Phi(\nabla u_\varepsilon)
 \big\langle
 \nabla_{\scriptscriptstyle\!\zeta} \Phi(\nabla u_\varepsilon), \,
\frac{ \nabla \Phi_0(z)}{|\nabla\Phi_0(z)|}  \big\rangle \, d\sigma(z)
\bigg| \leq
  \frac{n+\alpha+\varepsilon-1}{2}\, C_{{\scriptscriptstyle H},n,\alpha,b} \, R^{-\varepsilon} 
 \;\xrightarrow{R\rightarrow \infty} \; 0, \end{eqnarray*} where $ d\sigma(z) $ denotes the $n$-dimensional Lebesgue integration over the upper half boundary $ (\partial W_{\!R})^+:=\{z=(x,y)\in\mathbb R^{n+1}: \Phi_0(z)=R,\, y>0\} $ of the Wulff shape $ W_{\!R}:=\{z\in\mathbb R^{n+1}: \Phi_0(z)=R\}. $ Using now \eqref{eq20}, then integrating by parts in the domains  $ W_1^+, \,W_R^+\setminus W_1^+, $ and letting eventually $ R\rightarrow \infty, $ it follows, from the above observations, that \begin{flalign*} 
&\int_{  \mathbb R_+^{n+1} } y^\alpha \,\Phi^2(\nabla u_\varepsilon) \, dz
 =  
 \int_{  \mathbb R_+^{n+1} } y^\alpha \, \Phi(\nabla u_\varepsilon) \,\langle \nabla u_\varepsilon, \nabla_{\scriptscriptstyle\!\zeta}\!\Phi(\nabla u_\varepsilon)\rangle\, dz  =  \\ &
K(n,\alpha, b)\int_{ \mathbb R^{n}}
\frac{u_\varepsilon^2(x,0)}{  H_0^{1-\alpha}(x)  }\, dx
+
\left[  \frac{(\alpha + b -2)^2 - \varepsilon^2}{4} \right]
 \int_{ \mathbb R_+^{n+1} } \frac{  y^\alpha\, u_\varepsilon^2(z)}{\Phi_0^2(z)}\, dz
+\varepsilon  \int_{  (\partial W_{1})^+ } \frac{y^\alpha 
\, \psi^2(z) }{ |\nabla \Phi_0(z)|}\, \,d\sigma(z).     
\end{flalign*}
Then, letting $\varepsilon \rightarrow 0, $ we get 
$$\int_{  \mathbb R_+^{{n+1}} } y^\alpha \Phi^2(\nabla u_\varepsilon) \, dz
-
K(n,\alpha, b)\int_{  \mathbb R^{n} }
\frac{u_\varepsilon^2(x,0)}{ H_0^{1-\alpha}(x)  }\, dx
 -  \frac{(\alpha + b - 2)^2}{4}
 \int_{  \mathbb R_+^{{n+1}} } \frac{  y^\alpha\,  u_\varepsilon ^2(z)}{\Phi_0^2(z)}\, dz
\longrightarrow 0. $$
\section{Sharp infinite series improvement}
\label{section.4} In this section we will give the proof of Theorem \ref{Theorem2}. For notational convenience, we will assume that $ D=1, $ without loosing the generality, since \eqref{ineq15}  is invariant under scaling as implied by \eqref{eq18}. A crucial role in the proof, will play the function $ \psi_k, $ defined by
\begin{equation}\label{eq62}
 \psi_k(z) \,=\, \psi(z)\, X_1^{-1/2}(\rho)\, X_2^{-1/2}(\rho)\cdots
  X_k^{-1/2}(\rho)
  \,=\,
  \psi(z) P_k^{-1/2}(\rho),\quad \rho = \Phi_0(z),
\end{equation}
where $ \psi $ is defined in \eqref{eq37}-\eqref{eq38}. Exploiting the differentiation formula

\begin{equation}\label{eq63}
X_i'(\rho) = \frac{1}{\rho} X_1(\rho) \cdots X_{i-1}\, X_i^2(\rho) \,= \, \frac{1}{\rho} \, P_i(\rho)\, X_i(\rho), \;\; i=1,\ldots , k,
\end{equation} hence
\begin{equation}\label{eq64}
P'_k(\rho) \,= \,\frac{1}{\rho}\, P_k(\rho)\, S_k(\rho), \quad\mbox{with}\quad S_k(\rho) := \sum\limits_{i=1}^k P_i(\rho),
\end{equation} 
 it can be shown that $ \psi_k $ satisfies the Euler-Lagrange equations associated to \eqref{ineq15},
\begin{numcases}{\label{eq65}}
 \Delta_{\scriptscriptstyle \Phi,\alpha} \psi_k(z) + \frac{(\alpha + b-2)^2}{4}\frac{y^\alpha \psi_k(z)}{  \Phi_0^2(z) }
+ \frac{y^\alpha\,\psi_k(z)}{4 \Phi_0^2(z)}\sum_{i=1}^k P^2_i(\rho) = 0, 
\qquad { z \in U^+}, 
\label{65a}
\\
 \lim\limits_{y\rightarrow 0^+} y^\alpha \frac{ \partial \psi_k(x, y)}{\partial y}
=
- K(n, \alpha, b) \frac{\psi_k(x,0)}{  H_0^{1-\alpha}(x)},
    \qquad\qquad  x \in U_0 \setminus\{0\}.
\end{numcases}
We are now ready to proceed with the 
\begin{proof}
[\textbf{Proof of Theorem \ref{Theorem2}}]
By standard approximation, it suffices to prove \eqref{ineq15} for  $ u \in C^\infty_0(U\setminus\{0\}). $ Let us then integrate \eqref{eq47},  apply the divergence theorem to the last term 
and finally use (\ref{eq65}), to obtain 

\begin{eqnarray} 
\int_{U^+}  y^\alpha   \, f(u, \psi_k) \,dz
&=&
\int_{U^+}  \, y^\alpha \, \Phi^2(\nabla u)  \,dz
-K(n, \alpha, b)
\int_{U_0}   \frac{u^2(x,0)}{H_0^{1-\alpha}(x)} \,dx
-\frac{( \alpha + b - 2 )^2 }{4}
\int_{U^+} \frac{y^\alpha\,u^2(z)}{\Phi_0^2(z)} \,dz
\nonumber
\\
& &-
\frac{1}{4}\sum_{i=1}^{k}
\int_{U^+} 
  \, y^\alpha \, 
\frac{P_i^2\,u^2(z)}{\Phi_0^2(z)} \,dz, 
\label{eq66}
\end{eqnarray} where $ f $ is given in \eqref{eq48}, and recall from \eqref{eq49} that $ f(u,\psi_k)\geq0.$ We then conclude that 

\begin{eqnarray*} K(n,\alpha,b)\int_{U_0} \frac{u^2(x, 0)}{H_0^{1-\alpha}(x)} \,dx + \frac{ ( \alpha + b - 2 )^2 }{4} \int_{U^+} \frac{y^\alpha u^2{(z)}}{ \Phi_0^2(z)} \,dz + \frac{1}{4}\sum_{i=1}^{k} \int_{U^+} \frac{y^\alpha P_i^2\,u^2(z)}{ \Phi_0^2(z)} \,dz \leq \int_{U^+} y^\alpha {\Phi^2(\nabla u) } \,dz,
\end{eqnarray*} for all $  u\in C^\infty_0(U\setminus \{0\}), $ therefore by approximation, for all $ u \in C^\infty_0(U), $ and eventually  taking the limit as $ k\rightarrow \infty, $ we obtain inequality \eqref{ineq15}.

Let us next verify the optimality of the constants $\frac{1}{4} $ for the remainder terms appearing in \eqref{ineq15}.  To this aim, we fix $ k \in \mathbb N, $ and setting $ \varepsilon = (\varepsilon_0, \varepsilon_1,\ldots,\varepsilon_k), $ where $\varepsilon_0 > 0, \, \varepsilon_1 > 0, \ldots, \varepsilon_k>0, $ we will show that there exist functions $u_{\varepsilon} \in D^{1,2}(U_+,  y^\alpha dz) $ such that

\begin{eqnarray*}
Q_k[u_{\varepsilon}]:=
 \frac{\int\limits_{U^+} y^\alpha \,  \Phi^2(\nabla u_{\varepsilon})  \, dz 
-K(n,\alpha, b )\int\limits_{U_0} \frac{u_{\varepsilon}^2(x,0)\, dx }{H_0^{1-\alpha}(x)} -\frac{ ( \alpha + b - 2 )^2 }{4}\int\limits_{U^+} \frac{  y^\alpha u_{\varepsilon}^2(z)}{\Phi_0^2(z)} \, dz-\frac{1}{4} \sum\limits_{i=1}^{k-1}\int\limits_{U^+}
\frac{  y^\alpha P_i^2  u_{\varepsilon}^2(z)}{ \Phi_0^2(z)}\, dz}
{\int\limits_{U^+} \frac{  y^\alpha \, P_k^2 u_{\varepsilon}^2(z)}{\Phi_0^2(z)} \,dz}
\xrightarrow{\varepsilon \to 0} \frac{1}{4}.
\end{eqnarray*} In view of the ground state representation (\ref{eq66}), we have \begin{eqnarray*} Q_k[u_{\varepsilon}]
= \frac{ \int_{U^+} y^\alpha  f(u_\varepsilon, \psi_k)\,dz }
{ \int_{U^+} \frac{ y^\alpha \,P_k^2 \, u_{\varepsilon}^2(z)  }{\Phi_0^2(z)} \,dz }
+ \frac{1}{4}, \end{eqnarray*} where  (see \eqref{eq48})
\begin{equation}\label{eq67}
 f(u_\varepsilon,\psi_k):=\Phi^2(\nabla u_\varepsilon) +\frac{u_\varepsilon^2}{\psi_k^2} \Phi^2(\nabla\psi_k) -2\frac{u_\varepsilon}{\psi_k}\langle \nabla u_\varepsilon, \Phi(\nabla\psi_k)\nabla_{\!\zeta} \Phi(\nabla\psi_k)\rangle,
\end{equation}  hence it is sufficient to show that there exist functions $  u_\varepsilon \in D^{1,2}(U_+,  y^\alpha dz) $ such that 

\begin{equation}
\label{68}
\frac{ \int_{U^+} y^\alpha f(u_\varepsilon, \psi_k)\,dz  }
{
\int_{U^+} 
\frac{ y^\alpha \,P_k^2 \, u_{\varepsilon}^2(z)}{ \Phi_0^2(z)} \,dz  
}
\xrightarrow{\varepsilon \to 0} 0.
\end{equation}
Let $ \delta>0 $ such that  $ W_{\delta} \subset U. $ We fix some radius $ 0 < R < \delta/2 $ and define the functions 

\begin{equation*}
u_{\varepsilon}(z)
 =  
 \psi_k(z)\, \upsilon_\varepsilon(\rho) \, \eta(z),
\qquad
\mbox{where}
\qquad
\upsilon_\varepsilon(\rho):=\rho^{\varepsilon_0}X_1^{\varepsilon_1}(\rho) \cdots X_k^{\varepsilon_k}(\rho),
\;\;\; \rho = \Phi_0(z),
\end{equation*}
and 
$\eta \in C_0^\infty(W_{2 R})$ with $ \eta\equiv 1 $ in $W_{R}. $ 
 Let us now show that $u_{\varepsilon} $ indeed satisfy condition 
\eqref{68}. 

We begin with the denominator in
\eqref{68}. In view of \eqref{rel46}, we decompose the integral into $ W_R^+$ and $ W_{2R}^+ \setminus W_R^+, $ and then apply ``$\Phi_0$-polar coordinates'' $ z=\rho\mathrm{w}, \,\rho=\Phi_0(z),\,\Phi_0(\mathrm{w})=1, $ to deduce that, for some positive constants $ c_{{\scriptscriptstyle H},n,\alpha,b},\, C_{{\scriptscriptstyle H},n,\alpha,b}, $ independent of $ \varepsilon, $ 

\begin{eqnarray*}
\int\limits_{U^+} 
\frac{ y^\alpha \, P_k^2\, u_{\varepsilon}^2(z)}{\Phi_0^2(z)} \,dz
 \geq 
c_{{\scriptscriptstyle H},n,\alpha,b} 
\int\limits_{ W_R^+ } \frac{
y^\alpha \, P_k \, \upsilon_\varepsilon^2
}{\Phi_0^{n+\alpha+1 }} \,dz 
\;  + 
\!\!\int\limits_{ W_{{2R}}^+\setminus  W_R^+ }
\!\!\!
 \frac{y^\alpha 
P_k^2  \,u^2_\varepsilon 
}{ \Phi_0^2  } \,dz
 = 
C_{{\scriptscriptstyle H},n,\alpha,b}
\int\limits_0^R \frac{P_k \,\upsilon_\varepsilon^2}{\rho}\,d\rho + O(1),\qquad\mbox{as}\,\; \varepsilon\to 0.
\end{eqnarray*} We then take successively the limits   
$\varepsilon_0 \to 0, \ldots, \varepsilon_{k-1}\to 0,  $ and use eventually the relation  (\ref{eq63}) to obtain 

\begin{eqnarray}
\label{eq69}
 \lim\limits_
 {\varepsilon_{k}\rightarrow 0 }\cdots 
\lim\limits_
 {\varepsilon_1\rightarrow 0 } 
 \lim\limits_
 {\varepsilon_0\rightarrow 0 }
 \int_{U^+} 
\frac{  y^\alpha \, 
P_k^2\, u_{\varepsilon}^2(z) 
}{ \Phi_0^2(z)} \, dz
&\geq &
C_{{\scriptscriptstyle H},n,\alpha,b}
\lim\limits_{\varepsilon_{k}\rightarrow 0 }
\int_0^R \frac{P_{k}(\rho) \, X_k^{ 2 \varepsilon_k}(\rho)}{\rho}\,d\rho + O(1)\nonumber\\&=& 
C_{{\scriptscriptstyle H},n,\alpha,b} \lim\limits_
 {\varepsilon_k\rightarrow 0 } \frac{1}{2\varepsilon_k}X_k^{2\varepsilon_k}(R)+O(1)
= +\infty.
\end{eqnarray} Next, we proceed to estimate the numerator in \eqref{68}. Exploiting the specific form of $ u_\varepsilon, $ let us first derive an explicit expression of the integrand $ f(u_\varepsilon, \psi_k), $ in $ W_R $ where $ \eta\equiv1. $ Taking into account the $ H_0$-norm's properties \eqref{eq18}, \eqref{eq25}, and utilizing (\ref{eq45}), 
(\ref{eq64}) together with the fact that (cf. (\ref{eq43}))

$$
\nabla (\upsilon_\varepsilon(\rho))
=
\frac{ \upsilon_\varepsilon(\rho)
J_\varepsilon(\rho)}{\rho^2}
 \,  \bigl(H_0(x) \nabla H_0(x),  y\bigr),
\quad\mbox{where}\;\;
J_\varepsilon:=\varepsilon_0
+ 
\sum_{i=1}^k \varepsilon_i \, P_i,
$$ 
 we deduce that, for $ z \in W_R, $ 

\begin{equation}\label{eq70}
\Phi^2(\nabla u_\varepsilon) =  \frac{u_\varepsilon^2 J_\varepsilon}{\rho^2}  
\big(  J_\varepsilon   +    1-a-n-S_k  \big) + \frac{u_\varepsilon^2}{\psi_k^2} \Phi^2(\nabla \psi_k).
\end{equation}
Moreover, by (\ref{eq18})-(\ref{eq20}), (\ref{eq25}), (\ref{eq27}), (\ref{eq29}) we get, for $ z \in W_R, $

\begin{equation}\label{eq71}
\Phi(\nabla\psi_k) \, \bigl\langle \nabla u_\varepsilon , \nabla_{\scriptscriptstyle\!\zeta} \Phi(\nabla\psi_k)\bigr\rangle  
\;=  \frac{u_\varepsilon  \psi^2 J_\varepsilon }{2 \psi_k P_k \rho^2}
\;  \big(   1 - \alpha - n  -  S_k
 \big)  + \frac{u_\varepsilon}{\psi_k}\Phi^2(\nabla\psi_k).
\end{equation}
Substituting (\ref{eq70})-(\ref{eq71}) into (\ref{eq67}), we obtain that

$$ f(u_\varepsilon, \psi_k) \,=\, \rho^{2\varepsilon_0-2} \, \psi^2 \, 
\big( X_1^{2\varepsilon_1-1} X_2^{2\varepsilon_2-1} \cdots X_k^{2\varepsilon_k-1}(\rho)
\big) \, J_\varepsilon^2,\quad \mbox{in}\;\; W_R.  $$
As above, we split the integration ``near and far'' from the singularity at the origin, and taking into account \eqref{rel46}, we get, for some positive constant $ C_{{\scriptscriptstyle H},n,\alpha,b} $ independent of $ \varepsilon, $ 

$$
\int_{U^+} y^\alpha  f(u_\varepsilon, \psi_k) \,dz
\leq
\int_{W_R^+} y^\alpha \rho^{2\varepsilon_0-2} \, \psi^2 \, 
\big( 
X_1^{2\varepsilon_1-1} X_2^{2\varepsilon_2-1} \cdots X_k^{2\varepsilon_k-1}(\rho)
\big)
\, J_\varepsilon^2 \,dz 
+ O(1),\qquad\mbox{as}\;\;\varepsilon\to0.
$$
By monotone convergence, the integral in the right hand side vanishes as $\varepsilon\to0, $ therefore

\begin{equation}\label{72}
\int_{U^+} y^\alpha  f(u_\varepsilon, \psi_k) \,dz = O(1),\qquad\mbox{as}\;\;\varepsilon\to0.
\end{equation}
Then \eqref{68} follows from 
 \eqref{eq69} together with \eqref{72},  and the proof of the optimality of the constants $ \frac{1}{4} $ is now complete.

To complete the proof of Theorem \ref{Theorem2}, it remains to verify that 
the logarithmic weight of the remainder terms cannot be replaced by smaller powers of $ X_i. $ More precisely, for any $ k = 1, 2,\ldots, $ we consider the $(k-1)$-improved Finsler Hardy-trace Hardy functional

$$ 
I_k[u]:=
\int_{U^+} y^\alpha \, \Phi^2(\nabla u)  \, dz
-K(n,\alpha, b )
\!\!
\int_{U_0} \frac{u^2(x,0)\, dx }{ H_0^{1-\alpha}(x)} 
-\frac{ ( \alpha + b - 2 )^2 }{4}
\int_{U^+} \frac{  y^\alpha u^2}{ \Phi_0^2} \,dz
-\frac{1}{4} \sum\limits_{i=1}^{k-1}
\int_{U^+}
\frac{  y^\alpha P_i^2  u^2}{\Phi_0^2}\, dz
$$  where the summation $ \sum_{i=1}^{k-1} $ denotes zero if $ k=1, $ and we have to show that the following inequality fails for any $ 0 < \epsilon < 2,\; C>0, \, $

$$  C \int_{U^+} \frac{ y^\alpha\, P^2_{k-1}  X_{k}^{2-\epsilon}}{ \Phi_0^2(z)  }\, u^2(z)\,dz \leq I_k[u], \;\;\; \forall u \in C_0^\infty(U).
$$ We will confirm the claim by taking a sequence $ \{u_m\} \subset  D^{1,2}(U^+,  y^\alpha dz), $ such that 

$$
\frac
{I_k[u_m]
}
{
\int_{U^+} 
\frac{ y^\alpha \, P^2_{k-1}  X_{k}^{2-\epsilon} \,u_m^2(z)}{ \Phi_0^2(z) }\, dz }
\xrightarrow{m \to \infty} 0,
$$ which, in view of \eqref{eq66}, is equivalent to  showing that 
\begin{equation}\label{73}
\frac{N[u_m]}{D[u_m]}
:=
\frac{
\int_{U^+} y^\alpha  f(u_m, \psi_{k-1})  \, dz
}
{
\int_{U^+} 
\frac{ y^\alpha \,P_{k-1}^2\, X_k^{2-\epsilon} \,  u_m^2(z)  } {\Phi_0^2(z)} \, dz
}
\xrightarrow{m \to \infty} 0.
\end{equation}
Notice also that it suffices to prove  the claim, only for the case $ 0 < \epsilon < 1, $ since $ X_k^{2-\epsilon_0} > X_k^{2-\epsilon}, $ $ \forall \epsilon_0 > \epsilon. $ To determine such a sequence satisfying \eqref{73}, we take $ \epsilon < \epsilon_m < 1, $ with
\begin{eqnarray}\label{74}
 \epsilon_m \xrightarrow{m\to\infty}\epsilon, \quad\mbox{ so that }  \quad m^{\epsilon-\epsilon_m}<1/2, 
\end{eqnarray} and  define recursively the radii $ R_1(m) = e^{1-m},\; R_{j + 1}(m)=R_j(e^{m-1}),\, j=1, 2 , \ldots, k,  $ so that $$ X_{k}(\rho) = \frac{1}{m}  \Longleftrightarrow \rho = R_{k}(m). $$ We then consider the functions $ u_m=\eta \psi_{k-1} \upsilon_m, $ where $ \eta \in C_0^\infty(U)$ with $ \eta\equiv 1 $ in $  W_\delta,  $ and  \begin{eqnarray*}
 \upsilon_m(z)=
\begin{cases} 
X_{k}^\frac{\epsilon_m-1}{2}(\rho)
, \; &  R_{k}(m) \leq \rho, \\
  m^\frac{3-\epsilon_m}{2}X_{k}(\rho)
  , \; & \rho\leq R_{k}(m), 
        \end{cases} 
\qquad \rho=\Phi_0(z).
\end{eqnarray*}
After some straightforward manipulations, using (\ref{eq18}), (\ref{eq19}), (\ref{eq25}),  (\ref{eq27}), (\ref{eq29}),  jointly with (\ref{eq43}), (\ref{eq45}),  and utilizing  (\ref{eq63}),  (\ref{eq64}), we deduce that,  in $ W_\delta, $
$$
f(u_m, \psi_{k-1})  = 
\frac{1}{\rho^2} \, P_{k-1}^{-1} \, \psi^2 \,\phi_m^2,
\qquad 
\mbox{where}  \qquad 
\phi_m(\rho):=
\begin{cases}
\frac{\epsilon_m-1}{2} 
P_{k}X_{k}^\frac{\epsilon_m - 1}{2}(\rho), &   R_k(m) <  \rho < \delta, 
\\
m^\frac{3 - \epsilon_m}{2}\, P_{k}X_{k}(\rho),  & \rho <  R_k(m),
\end{cases}
$$ which together with the fact that  
 $ \psi  \sim \Phi_0^{-\frac{n+\alpha-1}{2}} $ (by \eqref{rel46}),  imply that, for some positive constant $ C_1 $ independent of $ m, $ 
\begin{equation}
N[u_m] \leq C_1 \int_{W_\delta^+ } 
\frac{  y^\alpha P_{k-1}^{-1} {\phi_m^2} }{ \Phi_0^{n+\alpha+1} }\,dz + O(1),
\qquad\mbox{as}\;\; m \to\infty.
 \end{equation}
By \eqref{rel46} as well we have that
\begin{equation}
D[u_m] \geq C_2 
\int_{W_\delta^+} \frac{ y^\alpha\, P_{k-1} X_{k}^{2-\epsilon}\,
{ \upsilon_m^2}}{ \Phi_0^{n+\alpha+1} }\,dz  + O(1),
\qquad\mbox{as}\;\; m \to\infty,
\label{76}
\end{equation}
for some positive constants $  C_2 $ independent of $ m. $ 
Then, we employ $\Phi_0$-polar coordinates as above, and next, on account of (\ref{eq63}), make a change of the $\rho$-variable,  to obtain
\begin{equation}\label{77}
\int_{W_\delta^+ } 
\frac{  y^\alpha P_{k-1}^{-1} \phi_m^2 }{\Phi_0^{n+\alpha+1} }\,d z 
=  C_{\!{\scriptscriptstyle H}\!,\delta,n,\alpha}
\left(
 \frac{(\epsilon_m - 1)^2}{4\epsilon_m }(X_k^{\epsilon_m}(\delta) - m^{-\epsilon_m})
+
\frac{m^{-\epsilon_m }}{3}
\right)
\end{equation}
and
\begin{equation}\label{78}
\int_{W_\delta^+} \frac{ y^\alpha\, P_{k-1} X_{k}^{2-\epsilon}\,
{\upsilon_m^2}}{\Phi_0^{n+\alpha+1}}\,dz =
C_{\!{\scriptscriptstyle H}\!,\delta,n,\alpha}
\left(
\frac{X_k^{\epsilon_m - \epsilon}(\delta) - m^{\epsilon-\epsilon_m}}{\epsilon_m - \epsilon}
 +  \frac{m^{\epsilon-\epsilon_m}}{3-\epsilon}
 \right),
\end{equation} where $ C_{\!{\scriptscriptstyle H}\!,\delta,n,\alpha} :=\int_{(\partial W_\delta)^+} y^\alpha\, \Phi(\frac{\nabla\Phi_0}{|\nabla\Phi_0|}) \, d\sigma(z). $ 
Finally, (\ref{73}) follows by taking the limit as $ m \to \infty $ in (\ref{77})-(\ref{78}), combined with (\ref{74})-(\ref{76}). This completes the proof of Theorem \ref{Theorem2}. \end{proof}
\quad\\
Let us next 
discuss the validity of Hardy and trace Hardy inequalities on a cone $ {\mathcal C}_\vartheta, $ with respect to the $\Phi$-metric, which is defined as

$$
{\mathcal C}_\vartheta = \big\{(x, y):\in\mathbb R^{n+1}:\;\; y > (\tan\vartheta) H_0(x)\big\},
$$
for some angle $ \vartheta\in\big(0, \frac{\pi}{2}\big). $ 

The extension of the Finsler-Kato-Hardy inequality \eqref{ineq4} to the case of the cone $ {\mathcal C}_\vartheta, $ with $\vartheta\neq0, $ states 
that for $ \,2 \leq b < n+1,  $ 

\begin{equation}\label{ineq79}
K_\vartheta(n,b)
\int_{\mathbb R^n}
\frac{u^2(x, (\tan\vartheta) H_0(x))}{ H_0(x)} \, dx
 + 
 \frac{(b -2)^2}{4} \int_{{\mathcal C}_\vartheta} \frac{  u^2(z)}{ \Phi_0^2(z)}\,dz
\leq
\int_{{\mathcal C}_\vartheta}  \Phi^2(\nabla u)\,dz,
\qquad
\forall u \in C_0^\infty(\mathbb R^{n+1}),
\end{equation}
with the best possible constant
\begin{multline*}
K_\vartheta(n, b) 
= 
\beta \,  \tan\vartheta
\,+\,
\Big[
 2\, a_1\, b_1\, (\tan\vartheta)
\,F(a_1+1,\, b_1+1;
\, \frac{3}{2};\, -\tan^2\vartheta)
+ \frac{C(n,b)}{2}
\,F(a_1 + \frac{1}{2},\,b_1 + \frac{1}{2};\, \frac{3}{2};\, -\tan^2\vartheta)
\\
-
\frac{2\,C(n,b)}{3} \, (a_1 + \frac{1}{2})\, (b_1 + \frac{1}{2}) \, (\tan^2\vartheta)
\,F(a_1+\frac{3}{2},\,b_1+\frac{3}{2};\, \frac{5}{2};\, -\tan^2\vartheta)
\Big]\;\widetilde{C}
\end{multline*}
where $ C(n,b) $ is given in \eqref{eq5},  
$ a_1:=\frac{5-n-b}{4},\, b_1=\frac{n+1-b}{4}, $
$ F$ stands for the hypergeometric function (see Appendix for the precise definition), and
\begin{equation*}
 \widetilde{C}
:=
2\,(\cos\vartheta)^{-2} \,
\Big[
 F(a_1, \,b_1;\, \frac{1}{2};\,-\tan^2\vartheta) 
\,-\, 
 C(n,b) \,(\tan\vartheta)\, F(a_1 + \frac{1}{2},\,b_1 + \frac{1}{2};\,\frac{3}{2};\,-\tan^2\vartheta)
\Big]^{-1}.
\end{equation*} 
%
%
%AVineq
%
It is clear that \eqref{ineq79} reduces to \eqref{ineq4}, for $\vartheta \to 0, $ as it is easily seen that 
$$ 
\lim\limits_{\vartheta\to0} K_\vartheta(n, b) = C(n,b).
$$
Inequality \eqref{ineq79} have been recently established by Alvino et al.  \cite{AlvinoTakahashi}, by adjusting their strategy of the half-space case to the cone's context; it is also easily seen to be proved by adapting the proof of Theorem \ref{Theorem1}. We may as well follow the argumentation of the proof of Theorem \ref{Theorem2} above, to get a refined version of \eqref{ineq79} with sharp correction terms, when $ {\mathcal C}_\vartheta $ is replaced by bounded conic domains. More precisely, we have the following extension of Theorem \ref{Theorem2} to the case of the domain $ \,U_\vartheta:=U\cap{\mathcal C}_\vartheta,\, $  the intersection of the infinite cone $ {\mathcal C}_\vartheta $  with  a bounded domain $ U $ containing the origin:

\begin{theorem}[Series-type improvement of interpolated Finsler-Kato-Hardy inequalities on cones]\label{Theorem3} Let $ \,2 \leq b < n+1. $  
Then the following inequality is valid for all $ u\in  C_0^\infty(U), $ 

\begin{equation} \label{80}
K_\vartheta(n,b) \int_{\mathbb R^n} \frac{u^2(x, (\tan\vartheta) H_0(x))}{ H_0(x)} \, dx +  \frac{(b-2)^2}{4} 
\int_{ U_\vartheta } \frac{  u^2(z)}{\Phi_0^2(z)}\; dz + 
\frac{1}{4}\sum_{i=1}^\infty \int_{ U_\vartheta } \frac{ P^2_i}{\Phi_0^2(z)} \, u^2(z) \;dz \leq \int_{ U_\vartheta } \Phi^2(\nabla u)\;dz,\;\,
\end{equation} where 
$ P_i=P_i(  \Phi_0(z)/D), $ with 
 $ D:=\sup\limits_{z\in  U_\vartheta}  \Phi_0(z). $ For any $ k = 1, 2, \cdots, $ and $ b \in [2, n+1), $ the constant $ \frac{1}{4} $ of the $ k$-th remainder term is the best possible,   i.e.
$$ 
\frac{1}{4}
=
\inf\limits_{u\in  C_0^\infty(U)}
\frac{
\int_{U_\vartheta} \Phi^2(\nabla u) \, dz
-
K_\vartheta(n, b)
\int_{\mathbb R^n}
\frac{u^2(x, (\tan\vartheta) H_0(x))}{H_0(x)}\,dx
- \frac{(b-2)^2}{4} \int_{U_\vartheta} \frac{ u^2(z)}{ \Phi_0^2(z)} dz
- \frac{1}{4}
\sum\limits_{i=1}^{k-1} \int_{U_\vartheta} \frac{  P^2_i}{\Phi_0^2(z)}\, u^2(z)\,dz
}
{
\int_{U_\vartheta} \frac{ P^2_k}{\Phi^2(z)}\, u^2\,dz
}.
$$
Moreover, for each $ i=1, 2, \ldots, $ the logarithmic weight $X_i^2 $ in \eqref{80} cannot be replaced by a smaller power of $X_i. $
\end{theorem}
As already mentioned, we may prove Theorem \ref{Theorem3}, by adopting the argumentation for the proof of Theorem \ref{Theorem2}. 
For the sake of completeness, let us outline the proof, yet omitting the detailed calculations, as they are straightforward generalizations of the previous ones.

To prove \eqref{80}, we consider again the function $ \psi_k $ defined in \eqref{eq62}, and notice that it satisfies the Euler-Lagrange 
equations associated to  \eqref{80},
\begin{numcases}{\label{81}} 
\operatorname{div}\big(\Phi(\nabla\psi_k) \,\nabla_{\!\zeta}\Phi(\nabla\psi_k)\bigr)
 + \frac{(b-2)^2}{4}\frac{\psi_k(z)}{\Phi^2_0(z)} + \frac{\psi_k(z)}{4  \Phi_0^2(z)}\sum_{i=1}^k P^2_i(\rho) = 0, 
\quad z\in {\mathcal C}_\vartheta, 
\label{81a}
\\
\big\langle
\hat{\eta}_{\vartheta}(x)
,\;
\Phi(\nabla\psi_k)\nabla_{\!\zeta}\Phi(\nabla\psi_k)
\big\rangle
=  
K_\vartheta(n,\beta)
\; 
m_{\vartheta}(x) 
\;
\frac{\psi_k(z) }{H_0(x)}
,\quad z=(x,(\tan\vartheta) H_0(x)),
\;\; x \in  \mathbb R^{n}\setminus\{0\},
\label{81b}
\end{numcases} where
\begin{equation*}
K_\vartheta(n,b)
=
\beta \tan\vartheta
-
\frac{  B'(\tan\vartheta)}{B(\tan\vartheta) \cos^2\vartheta}
\end{equation*}
and the vector $$\hat{\eta}_\vartheta(x):=m_{\vartheta}(x)\, \big( (\tan\vartheta) \nabla H_0(x), \; - 1 \big),
\qquad\mbox{with}\qquad\quad
m_{\vartheta}(x):=\frac{1}{\sqrt{1+\tan^2\!\vartheta |\nabla H_0(x)|^2}},$$
is  the unit outward normal vector on the boundary 
$$
\partial {\mathcal C}_\vartheta = \{z=(x,(\tan\vartheta) H_0(x)),\;\forall x\in\mathbb R^n\}.
$$
Indeed, the PDE \eqref{81a} is nothing but the PDE (\ref{65a}), with $\alpha=0 $ there,
which we have already verified, while 
(\ref{81b}) can be straightforwardly checked by exploiting the Finsler-norms' properties: (\ref{eq18}), (\ref{eq19}),  (\ref{eq25}), (\ref{eq27}), (\ref{eq29}).

Again, by standard approximation, it suffices to prove \eqref{80} for  $ u \in C^\infty_0(U\setminus\{0\}). $
Then, for any $ k, $ we apply \eqref{eq47} for such $ u $ with $ \psi_k, $ integrate and
 apply the divergence theorem to the last term,  and finally use equations \eqref{81}, to obtain $$
0\leq  \int_{U_\vartheta} f(u, \psi_k) \,dz  =    
\int_{U_\vartheta}   \Phi^2(\nabla u)  \,dz
- K_\vartheta(n, b) \int_{\partial {\mathcal C}_\vartheta}   \frac{u^2}{H_0} \, m_\vartheta\, d\sigma  - 
\frac{( b - 2 )^2 }{4} \int_{U_\vartheta} \frac{u^2}{\Phi_0^2} \,dz  \, - \frac{1}{4} \sum_{i=1}^{k} \int_{U_\vartheta}  \frac{P_i^2\,u^2}{\Phi_0^2} \,dz,   $$ which proves \eqref{80}. Here $ d\sigma $ denotes the $n$-dimensional integration over the boundary $ \partial {\mathcal C}_\vartheta, $ i.e. $ d\sigma= dx/m_\vartheta. $  Furthermore, for each $ k, $ the optimality of the constant $ \frac{1}{4} $ of the remainder term, as well as the optimality of logarithmic weight's power, are again demonstrated, as in the proof of Theorem \ref{Theorem2},  by taking suitable energetic approximations of the ground state $ \psi_k. $ 
\section*{Appendix} As mentioned in Section  \ref{subsection.2.2}, this appendix provides a detailed proof of the properties of $ B $ stated therein. For completeness' sake, and so that the context becomes more self-contained, we begin with a brief summary of those properties of hypergeometric equations and their solutions, that we use thereafter in surveying $ B. $ 
\paragraph{Hypergeometric equation}
In what follows, we refer to  \cite[\S 15]{Abramowitz.handbook}, \cite[Chap. II]{Erdelyietal} and \cite[\S\S 2.1.2-2.1.5]{Polyanin.Zaitsev.Handbook.Exact.Solutions.for.O.D.E}, among many other texts regarding hypergeometric equations.

For a complex function $\omega$ of the complex variable $\mathrm{z},$ let us consider the hypergeometric differential equation
\begin{equation}\label{82}
\mathrm{z} \,(1-\mathrm{z})\,  \frac{d^2\omega}{d\mathrm{z}^2} \,+\, [ \,\crc \,- \,( \ara\, +\, \brb\, + \,1) \mathrm{z}\, ] \, \frac{d\omega}{d\mathrm{z}}\, 
- \, \ara \,\brb \, \omega = 0
\end{equation} where $\ara, \brb, \crc \in \mathbb R $ such that
\begin{equation}\label{83}
\crc - \ara - \brb \geq 0, \qquad \brb>0, \qquad \crc > 0.
\end{equation}

The general solution of (\ref{82}), defined in the complex domain cut along the interval $[1,\infty ) $ of the real axis, is given by (see \cite[15.5.3, 15.5.4]{Abramowitz.handbook}) \begin{equation}\label{84}
\omega(\mathrm{z}) = \mathcal{C}_1 \, F(\ara, \,\brb;\, \crc;\,\mathrm{z}) + 
 \mathcal{C}_2 \,\mathrm{z}^{1-\crc}\, F(\ara-\crc+1,\, \brb-\crc+1;\, 2-\crc;\, \mathrm{z}), 
\end{equation}for arbitrary complex constants $\mathcal{C}_1, \mathcal{C}_2.$
Here the hypergeometric function $ F(\ara,\, \brb;\, \crc;\, \mathrm{z}) $
 is defined by the Gauss' series \cite[15.1.1]{Abramowitz.handbook}
\begin{equation}\label{85}
F(\ara,\, \brb;\, \crc;\, \mathrm{z})
=\sum\limits_{k=0}^\infty \frac{(\ara)_k (\brb)_k}{(\crc)_k} \frac{\mathrm{z}^k}{k!},
\end{equation} in the disk $|\mathrm{z}|<1, $ and by analytic continuation in the whole of complex plain cut along the interval $[1,\infty ) $ of the real axis. In (\ref{85}) we use the usual abbreviation $(\ara)_k = \ara (\ara+1) \cdots (\ara+k-1)$ and $(\ara)_0=1. $ Obviously, there holds 
$$ F(\ara,\, \brb;\, \crc;\, \mathrm{z}) = F(\brb,\, \ara;\, \crc;\, \mathrm{z}).$$

For later use, let us give an explicit expression of 
the analytic continuation of the series (\ref{85}) into the domain $\{\mathrm{z}\in\mathbb C: |\mathrm{z}|>1,\, \mathrm{z}\not\in (1,\infty) \}. $ To this aim, we assume $|\mathrm{z}|>1,\; \mathrm{z}\not\in (1,\infty) $ and distinguish the following 
 cases.
\\
\\
\noindent
\textit{Case I: } If none of the numbers $\ara, \, \crc-\brb,\, \ara-\brb $ equals a non positive integer $m=0, -1, -2,\dots,$ then we have (see \cite[15.3.7]{Abramowitz.handbook})

\begin{eqnarray}\label{86}
F(\ara,\, \brb;\, \crc;\, \mathrm{z})&=&\frac{\Gamma(\crc) \Gamma(\brb-\ara) }{\Gamma(\brb) \Gamma(\crc-\ara)} 
(-\mathrm{z})^{-\ara}
F(\ara,\, \ara-\crc+1;\, \ara-\brb+1;\, \frac{1}{\mathrm{z}})
 \nonumber
\\
& & 
+\,
\frac{\Gamma(\crc) \Gamma(\ara-\brb) }{\Gamma(\ara) \Gamma(\crc-\brb)} 
(-\mathrm{z})^{-\brb}
F(\brb,\, \brb-\crc+1;\, \brb-\ara+1;\, \frac{1}{\mathrm{z}}).
\end{eqnarray}
\noindent
\textit{Case II: } If $\ara=\brb \neq -m, \; \forall m=0, -1, -2,\dots, \,$ and 
$\crc-\ara \neq l,\; \forall l = 1,  2,\ldots, $ then we have (see \cite[15.3.13]{Abramowitz.handbook}) 

\begin{equation}\label{87}
F(\ara,\, \ara;\, \crc;\, \mathrm{z})=\frac{\Gamma(\crc)   (-\mathrm{z})^{-\ara}}{\Gamma(\ara) \Gamma(\crc-\ara)}
\sum\limits_{k=0}^\infty \frac{(\ara)_k (1-\crc+\ara)_k }{(k!)^2} \,\mathrm{z}^{-k}
\big[
\ln(-\mathrm{z})+ 2 \Psi(k+1) -\Psi(\ara+k) - \Psi(\crc-\ara-k)
\big],
\end{equation}
where $ \Psi $ stands for the logarithmic derivative of the Gamma function, that is  $ \Psi(\mathrm{z})= -  \upgamma - \sum\limits_{k=0}^\infty\left(\frac{1}{\mathrm{z}+k}-\frac{1}{k+1}\right) $ with the Euler's constant $ \upgamma\; \approx 0.5772156649. $
\\
\\
\noindent
\textit{Case III: } Let us now consider the case where $\brb-\ara=m,\, $  $ m=1, 2,\dots,$ and $\ara \neq -k, \, \forall k=0,1,2,\dots $ 

If  $\crc-\ara \neq l,\,\forall l = 1, 2, \ldots,\, $ then 
we have (see \cite[15.3.14]{Abramowitz.handbook})
\begin{eqnarray}\label{88}
F(\ara,\, \ara+m;\, \crc;\, \mathrm{z})&=&\frac{\Gamma(\crc)   (-\mathrm{z})^{-\ara-m}}{\Gamma(\ara+m) \Gamma(\crc-\ara)}
\sum\limits_{k=0}^\infty \frac{(\ara)_{k+m} (1-\crc+\ara)_{k+m} }{(k+m)! \, k!} \mathrm{z}^{-k}
\big[
\ln(-\mathrm{z})
+ \Psi(1+m+k) + \Psi(1+k) 
\nonumber\\
& & -\Psi(\ara+m+k) - \Psi(\crc-\ara-m-k)
\big] + (-\mathrm{z})^{-\ara} \frac{\Gamma(\crc)}{\Gamma(\ara+m)} \sum\limits_{k=0}^{m-1} \frac{\Gamma(m-k)(\ara)_k}{k! \Gamma(\crc-\ara-k)} \mathrm{z}^{-k}.\end{eqnarray}

On the other hand, if $ \crc-\ara = l, $ where $ l=1,2,\ldots, $ such that $ l>m, $ then we have by \cite[(19) in \S2.1.4]{Erdelyietal},

\begin{flalign}\label{89}
 F(\ara,\, \ara+m;\, \ara+l;\, \mathrm{z}) & =
\frac{ \Gamma(\ara+l) }{ \Gamma(\ara+m) } (-\mathrm{z})^{-\ara}
\Bigg[(-1)^l (-\mathrm{z})^{-m}
\sum\limits_{k=l-m}^\infty \frac{(\ara)_{k+m} (k+m-l)! }{(k+m)!\,  k!} \,  \mathrm{z}^{-k}
\nonumber \\
&\quad 
+ \sum\limits_{k=0}^{m-1} 
\frac{(m-k-1)!\, (\ara)_k}{(l-k-1)!\, k! }\, \mathrm{z}^{-k} +
\frac{(-\mathrm{z})^{-m}}{(l-1)!}
\sum\limits_{k=0}^{l-m-1} \frac{(\ara)_{k+m} (1-l)_{k+m} }{(k+m)! \, k!} \, \mathrm{z}^{-k}\,\times
\nonumber
\\
&\qquad
\times\Big[\ln(-\mathrm{z})
 + \Psi(1+m+k) + \Psi(1+k) 
 -\Psi(\ara+m+k) - \Psi(l-m-k)
\Big]\Bigg].
\end{flalign}
\noindent
\textit{Case IV: } If at least one of the numbers $\ara, \crc-\brb $ equals a non-positive integer,  then $ F(\ara, \brb; \crc; \mathrm{z}) $ becomes an elementary function of $\mathrm{z}. $ 
More precisely, if $\ara=-m  $ with $m=0,1,2,\dots $ then the hypergeometric series in (\ref{85}) reduces to the polynomial (see \cite[15.4.1 ]{Abramowitz.handbook})
\begin{equation}\label{90}
F(-m,\, \brb;\, \crc;\, \mathrm{z})=\sum\limits_{k=0}^m \frac{(-m)_k (\brb)_k}{(\crc)_k} \frac{\mathrm{z}^k}{k!}.
\end{equation}On the other hand, if $\crc -  \brb =  -l,\; $
with $ l = 0, 1,  2,\ldots,\, $ then $ F(\ara,\, \brb;\, \crc;\, \mathrm{z}) $ is written in the form (see \cite[15.3.3]{Abramowitz.handbook}) \begin{eqnarray}\label{91}
F(\ara, \brb; \crc;\, \mathrm{z}) =  (1-\mathrm{z})^{-\ara-l} F(\crc-\ara,\, -l;\, \crc ;\, \mathrm{z}),
\end{eqnarray} where the hypergeometric function in the right hand side is a polynomial of degree $l, $ according to (\ref{90}).

We conclude this summary with the following useful formula, when differentiating $ F(\ara,\, \brb;\, \crc;\, \cdot) $ (see \cite[15.2.1]{Abramowitz.handbook})
\begin{equation}\label{92}
\frac{d}{d\mathrm{z}}F(\ara,\, \brb;\, \crc;\, \mathrm{z}) = \frac{\ara \, \brb}{\crc}\; 
F(\ara+1,\, \brb+1;\, \crc+1;\, \mathrm{z}).
\end{equation}
\paragraph{Asymptotics of $ \bm{B} $}  In this paragraph, we prove the positivity and monotonicity of $ B,  $ as well as the asymptotics \eqref{eq40}-\eqref{eq42}.
For notational convenience, hereafter we abbreviate 
\begin{flalign*}
& & & a_1  =  \frac{5-n-b}{4}, &  
& & & a_2 = a_1-c_1+1 = \frac{7 - n - 2 \alpha - b}{4}, & 
& & & c_1=\frac{\alpha+1}{2}, & 
\\
& & & b_1 = -\frac{\beta}{2} = \frac{n+1-b}{4}, &
& & & b_2 = b_1-c_1+1 = \frac{3+n-2\alpha -b}{4}, & 
& & & c_2 = 2 - c_1 = \frac{3-\alpha}{2}.&
\end{flalign*}
%
%
%. 
It is straightforward to verify that both sets of parameters 
$\; \{a_1, \,b_1, \,c_1\},\, $ $\{a_2, \,b_2, \,c_2\},\; $ satisfy the conditions  (\ref{83}). Therefore, we can apply  the theory, that is presented in the previous paragraph, wherever is needed. 

Firstly, we will derive an explicit expression of $ B(t)=\omega(\mathrm{z}), $ from which will result all its properties. 
Equation (\ref{eq39a}) belongs to the class of hypergeometric equations and according to (\ref{84}), the general solution is given by

\begin{equation}\label{93}
 \omega(\mathrm{z}) \,=\, C_1 \, F(a_1, \,b_1;\, c_1;\,\mathrm{z}) \,+\, 
 C_2 \,(-\mathrm{z})^{1-c_1}\, F(a_2,\,b_2;\,c_2;\,\mathrm{z}),\qquad \mathrm{z}\leq 0,
\end{equation} for some constants $ C_1,\, C_2, $ which will be specified by the boundary conditions. Applying condition (\ref{eq39b}) to (\ref{93}), and noting that $ F(a_1,\, b_1;\, c_1;\; 0) = F(a_2,\, b_2;\, c_2;\; 0) = 1, $ we obtain $ C_1=1. $ 

The constant $ C_2 $ will be evaluated by the condition at $\infty, $ so we need an expression for $\omega(\mathrm{z}),\, $ when $ \mathrm{z}<-1. $ To this end, we will distinguish the cases for $n, \alpha, b, $ corresponding to the formulas (\ref{86}) - (\ref{91}), which give the explicit expression for the hypergeometric functions in (\ref{93}). In all cases, we will show that 

\begin{equation}\label{94}
C_2 = - \frac{\Gamma(c_1)\Gamma(b_2)\Gamma(c_2-a_2)}{\Gamma(c_2) \Gamma(b_1) \Gamma(c_1-a_1)},
\end{equation} as well as the following asymptotics, \begin{eqnarray}\label{95}
\omega(\mathrm{z})=O\left((-\mathrm{z})^{-b_1}\right), \;\;\mbox{as}\;\; \mathrm{z}\rightarrow -\infty.
\end{eqnarray} In order to prove the claims (\ref{94})-(\ref{95}), we assume that $ \mathrm{z}<-1, $ and distinguish the following cases.
\\
\\
\noindent
\textit{Case I:} Assume that none of the numbers $a_1,\, $ $a_2,\, $ $c_1-b_1,\, $ $c_2-b_2,\, $ $a_1-b_1=$ $a_2-b_2=$ $\frac{2-n}{2}, \,$
is equal to a non positive integer. 
 Then, we substitute the expressions of $ F(a_1, \,b_1;\, c_1;\,\mathrm{z}), $ 
 $ F(a_2, \,b_2;\, c_2;\,\mathrm{z}), $ given by (\ref{86}), into (\ref{93}), next multiply by $(-\mathrm{z})^{b_1}, $  and note that $ b_2-c_2+1=b_1,\, a_2-c_2+1=a_1, $
to arrive at

\begin{eqnarray}\label{96}
 (-\mathrm{z})^{b_1}\omega(\mathrm{z}) 
 &=& (-\mathrm{z})^\frac{n-2}{2}
\left[\frac{\Gamma(c_1) \Gamma(b_1-a_1)}
{\Gamma(b_1) \Gamma(c_1-a_1)}
+ C_2  \frac{\Gamma(c_2) \Gamma(b_2-a_2)}
{\Gamma(b_2)  \Gamma(c_2-a_2) } 
 \right]   F(a_1, \,a_2; \,\frac{4-n}{2};\,\frac{1}{\mathrm{z}})
\nonumber \\ & & + \left[\frac{\Gamma(c_1) \Gamma(a_1-b_1)}
{\Gamma(a_1) \Gamma(c_1-b_1)}
+ C_2  \frac{\Gamma(c_2) \Gamma(a_2-b_2)}{\Gamma(a_2)  \Gamma(c_2-b_2) } 
 \right]  F(b_1, \,b_2; \,\frac{n}{2};\,\frac{1}{\mathrm{z}}).
\end{eqnarray}
For $ n>2, $ we combine (\ref{96}) with  (\ref{eq39c}) to deduce (\ref{94}); for $ n=1 $ (and thus $\alpha>0$) we again arrive at (\ref{94}) by assuming (\ref{eq41}).  Then (\ref{95}) results upon a substitution of (\ref{94})  in (\ref{96}).
\\
\\
\noindent
\textit{Case II: } Next we proceed with the case where $ a_1-b_1=a_2-b_2=0, $ that is $n=2. $  In this case the numbers $a_1,\, $ $a_2,\, $ $c_1-b_1,\, $ $c_2-b_2\, $ are positive non integers. Then, the explicit expression for the functions 
$ F(a_1, a_1;\, c_1;\,\mathrm{z}) $ 
and 
$F(a_2,\, a_2;\, c_2;\,\mathrm{z}) $ appearing in (\ref{93}), is given by (\ref{87}). 
Substituting (\ref{87}) into (\ref{93}) and then multiplying  by $(-\mathrm{z})^{b_1}, $ taking into account the relations  $a_2=1-c_1+a_1 $ and $a_1=1-c_2+a_2, $ we arrive at \begin{multline}\label{97}
(-\mathrm{z})^{b_1}\omega(\mathrm{z}) = \left[
\frac{\Gamma(c_1)}{\Gamma(a_1)\Gamma(c_1-a_1)} + \frac{C_2 \Gamma(c_2)
}{\Gamma(a_2)\Gamma(c_2-a_2)}
\right]  \ln(-\mathrm{z}) \sum\limits_{k=0}^\infty
 \frac{(a_1)_k (a_2)_k}{(k!)^2}
\mathrm{z}^{-k} + \sum\limits_{k=0}^\infty
 \frac{(a_1)_k (a_2)_k}{(k!)^2}
\, \mathrm{z}^{-k} \times \\ \times
\bigg[ \frac{\Gamma(c_1) \big(2\Psi(k+1) - \Psi(a_1+k)- \Psi(c_1 -a_1-k)\big)}{ \Gamma(a_1)  \Gamma(c_1 - a_1) } + C_2 \; \frac{\Gamma(c_2)  \big(2\Psi(k+1) - \Psi(a_2+k)- \Psi(c_2-a_2-k)\big)}{ \Gamma(a_2)  \Gamma(c_2 - a_2) } \bigg] \end{multline}
Then (\ref{97}) combined with (\ref{eq39c}), yield  
 (\ref{94})  with $n=2 $ there,
 and for this value of $ C_2, $ (\ref{97}) also implies (\ref{95}). 
\\
\\
\noindent
\textit{Case III:} We consider now the case that none of the numbers $a_1,\, a_2,\, $ $ c_1-b_1,\, $ $ c_2-b_2\, $  is equal to a nonpositive integer and  $b_1-a_1=b_2-a_2 = m, $ that is $ n = 2m + 2,\, $ for $ m = 1, 2,... $  

We first assume that the numbers $c_1 - a_1, \,$  $ c_2 - a_2\, $  are not equal to a positive integer. Then, the explicit expression for the functions $ F(a_1, \,a_1+ m ; \, c_1;\,\mathrm{z}) $ and $ F(a_2,\,a_2+m ;\, c_2;\,\mathrm{z}) $ appearing in (\ref{93}), is given in (\ref{88}). Therefore, we substitute (\ref{88}) into (\ref{93}) and then multiply  by $(-\mathrm{z})^{b_1}, $ to arrive at \begin{multline}\label{98}
(-\mathrm{z})^{b_1} \omega(\mathrm{z}) = \left[
 \frac{\Gamma(c_1)  }
 {\Gamma(a_1+m)\Gamma(c_1-a_1)}  + \frac{C_2 \Gamma(c_2)  }
 {\Gamma(a_2+m)\Gamma(c_2-a_2)}
 \right] \bigg[\ln(-\mathrm{z})  \, \sum\limits_{k=0}^\infty
 \frac{(a_1)_{k+m} (a_2)_{k+m}}{k!(k+m)!}
\,\mathrm{z}^{-k} \\ + (-\mathrm{z})^{m}
 \sum\limits_{k=0}^{m-1} \frac{\Gamma(m-k)(a_1)_k(a_2)_k}{k!} (-\mathrm{z})^{-k}
\bigg] + O(1),\qquad\mbox{as}\;\;\mathrm{z}\rightarrow -\infty.\end{multline} 
Then, condition (\ref{eq39c}) combined with \eqref{98} yields again the value (\ref{94}) of $ C_2, $ so that zero out the coefficient in the brackets above, whereafter (\ref{95}) follows immediately. 

If at least one of the numbers 
$c_1- a_1,\, c_2 - a_2 $ is equal to an integer $l=1,2,\ldots, $
 then we can use the formula (\ref{89}) to get the expression of $ F(a_1, \,a_1+m; \, a_1+l;\,\mathrm{z}) $ 
 or  
$ F(a_2, \,a_2+ m; \, a_2 + l;\,\mathrm{z}), $ respectively. Arguing as above, we derive again (\ref{94}) and (\ref{95}).
\\
\\
\noindent
\textit{Case IV: } We conclude with the case where at least one of the numbers $a_1,\, $ $a_2,\, $ $c_1-b_1,\, $ $c_2-b_2\, $  is equal to a non-positive integer. 
We will consider only the case where $a_1=-m $ for some $m=0,1,2,\ldots, $ while 
none of the two numbers $c_1-b_1, $ $c_2-b_2 $ is a non-positive integer. The argumentation for the other cases is quite similar. 
 
The expressions of the first and second hypergeometric function in 
(\ref{93}) are given by formulas 
 (\ref{90}), (\ref{86}), respectively. Substituting (\ref{86}), (\ref{90}) into (\ref{93}) and then multiplying  by $(-\mathrm{z})^{b_1}, $ we arrive at

\begin{eqnarray}\label{99}
 (-\mathrm{z})^{b_1}\omega(\mathrm{z}) &=& (-\mathrm{z})^\frac{n-2}{2}
\bigg[\frac{(b_1)_m}{(c_1)_m} + C_2 \, \frac{\Gamma(c_2) \Gamma(b_2-a_2)}
{\Gamma(b_2)  \Gamma(c_2-a_2) }
\bigg] \, F(a_2, \, -m ; \,\frac{4-n}{2};\, \frac{1}{\mathrm{z}}) \nonumber \\
& & +\,  C_2 \, \frac{\Gamma(c_2) \Gamma(a_2-b_2)}
{\Gamma(a_2) \Gamma(c_2-b_2)  } \;
F(b_2, \,b_1; \, \frac{n}{2};\, \frac{1}{\mathrm{z}}),
\end{eqnarray}
where we have used that $ F(-m,\, \mathrm{b}_1;\, \mathrm{c}_1;\, \mathrm{z})=\frac{(-\mathrm{z})^m \,(b_1)_m}{(c_1)_m}\;F(a_2,\,-m,;\;a_2-b_2+1;\;\frac{1}{\mathrm{z}}). $ 
Since $ n>2, $ condition (\ref{eq39c}) yields (\ref{94}), thereafter (\ref{95}) results upon a substitution of this value of $C_2, $ in (\ref{99}).

The proof of (\ref{94}), (\ref{95}), is now completed. 
At this point, we have completely determined the solution $ \omega $
for any $ n, \alpha, b, $  subject of course to the assumption of Theorems \ref{Theorem1}-\ref{Theorem3}, i.e. $ 2-\alpha\leq b <n+1,\, \alpha\in(-1, 1). $ 
 
We are now in position to compute the limit 
$ K(n,\alpha, b):=-\lim\limits_{t\rightarrow 0^+}t^\alpha B'(t)= 
2\lim\limits_{\mathrm{z}\rightarrow 0^-}(-\mathrm{z})^\frac{\alpha + 1}{2} \omega'(\mathrm{z}). $ 
To this aim, we differentiate (\ref{93}), exploiting (\ref{92}),  to obtain
\begin{flalign*}
\omega'(\mathrm{z}) 
\, =\,  \frac{a_1\, b_1}{c_1}
\,F(a_1+1,\, b_1+1;
\, c_1+1;\, \mathrm{z})
& \,- \,
C_2 \frac{1-\alpha}{2}\,(-\mathrm{z})^{-\frac{\alpha+1}{2}}
\,F(a_2,\,b_2;\, c_2;\, \mathrm{z})
\\
& \, - \, 
C_2 \,\frac{a_2\, b_2}{c_2}\,(-\mathrm{z})^\frac{1-\alpha}{2}
\,F(a_2+1,\,b_2+1;\, c_2+1;\, \mathrm{z}),
\end{flalign*}
then we take the limit as $ \mathrm{z}\to0^-, $  to conclude that
\begin{equation*} 
K(n,\alpha, b)
=
2\lim\limits_{\mathrm{z}\rightarrow 0^-}(-\mathrm{z})^\frac{\alpha + 1}{2} \omega'(\mathrm{z})
=-(1-\alpha) C_2
\end{equation*}
which, in view of the specified value  (\ref{94}) of $ C_2, $  proves (\ref{eq42}).

Let us show now the positivity and monotonicity of $B. $ 
We first assume that $5-n-b<0. $
Then the positivity of $B$ follows from the fact that if there exist $t_0>0$  such that $B(t_0)=0, $ then since $\lim\limits_{t\to \infty} B(t)=0 , $ there exists $t_m> t_0$ where $B$ attains local non-negative maximum or local non-positive minimum, which contradicts the ode (\ref{eq38a}). Therefore $B$ is positive and the same argument shows that $B$ is decreasing. 

If $5-n-b \geq 0, $  then we make the substitution $ g(t)=(1+t^2)^{b_1} B(t) $ which  transforms problem  (\ref{eq38}) to  \begin{numcases}{\label{100}}
 t\, (1+t^2)^2\,  g''(t) \, + \, [\alpha \,+\, (3-n)\,t^2] \, (1+t^2) \, g'(t) \, + \,  
 \frac{(n+1-b)(3-n -2\alpha -b)}{4}\, t \, g(t)=0,\;\;  t>0,\qquad \label{100a} \\
g(0)=1, \\ \lim_{t\rightarrow\infty}g(t) \in {\mathbb R}_+.\label{100c} \end{numcases} Notice here, that the positivity of limit (\ref{100c}),  follows directly from the explicit expression of $ B(t)=\omega(z).  $ Indeed, in  \textit{Case I}, we exploit the Euler's reflection formula
$$\Gamma(s)\, \Gamma(1-s)=\frac{\pi}{\sin(s \pi)} , \qquad s\in\mathbb R,$$
and we use in addition the reflection formula
$$\Psi(1-s)=\Psi(s) + \pi \,\cot(s \pi ),$$
in \textit{Cases II-III}, or the relation
$$\Gamma(s-m)\;=\;(-1)^{m-1}\;\frac{\Gamma(-s)\;\Gamma(1+s)}{\Gamma(m+1-s)},$$
in \textit{Case IV}, to get that
$$ \lim_{t\rightarrow\infty}g(t) = \frac{1}{\pi}
\; \Gamma(\mathrm{c}_1) \;
\Gamma(1-\mathrm{a}_1) \; 
\frac{\Gamma(\mathrm{b}_2)}{
\Gamma\big(\frac{n}{2}\big)} \; \sin\big(\frac{1-\alpha}{2}\,\pi\big) >0. $$

Now we can apply a minimum principle argument to problem \eqref{100}, to get the non negativity of $ g. $ Indeed, if there exists $ t_0>0 $ such that $ g(t_0)<0, $ then since  $ g(0)=1,\, $ 
 $\lim\limits_{t\to \infty} g(t)\geq 0 , $ there exists $ t_m> t_0 $ where $ g $ attains local negative minimum, which contradicts the ode (\ref{100a}). It follows that $ g $  is non negative, hence $B$ is  non negative. Then (\ref{eq58})  together with the negativity of $ B' $ in a neighborhood of the origin, yield the monotonicity and positivity of $B.$
 
The asymptotics for $ B $ in \eqref{eq40} follow by conditions (\ref{eq38b})-(\ref{eq38c}) together with the positivity of $B. $ As regards the asymptotics for $B' $ in \eqref{eq40}, we differentiate the expression (\ref{93}) using the differentiation formula (\ref{92}). 
Finally, \eqref{eq41} is directly verified, by substituting the explicit expression of $ B(t)=\omega(-t^2) $ via \eqref{96}-\eqref{99}\textemdash depending on the cases therein\textemdash and its corresponding derivative.

\end{document}